\documentclass[11pt]{amsart}
\usepackage{amsmath,amsthm,amsfonts,amssymb}
\usepackage{hyperref}
\usepackage[all]{xy}
\usepackage{graphicx}
\usepackage{amscd}
\usepackage{graphicx}
\usepackage{pb-diagram}
\usepackage{pstricks}
\usepackage[all]{xy}

\numberwithin{equation}{section}

\newtheorem{theorem}{Theorem}[section]
\newtheorem{definition}[theorem]{Definition}

\newtheorem{lemma}[theorem]{Lemma}
\newtheorem{remark}[theorem]{Remark}
\newtheorem{prop}[theorem]{Proposition}
\newtheorem{example}[theorem]{Example}
\newtheorem{conjecture}[theorem]{Conjecture}

\newcommand{\CG}{\mathcal G}
\newcommand{\CH}{\mathcal H}

\newcommand{\RR}{\mathbb{R}}

\newcommand{\CC}{\mathbb{C}}

\newcommand{\ZZ}{\mathbb{Z}}

\newcommand{\WT}[1]{\widetilde{#1}}

\renewenvironment{proof}[1][\proofname. ]{ { \noindent \it #1}}{\qed \\}

\begin{document}
\author[Cho]{Cheol-Hyun Cho}
\address{Department of Mathematical Sciences, Research institute of Mathematics\\ Seoul National University\\ San 56-1, 
Shinrimdong\\ Gwanakgu \\Seoul 47907\\ Korea}
\email{chocheol@snu.ac.kr}
\author[Hong]{Hansol Hong}
\address{Department of Mathematical Sciences \\ Seoul National University\\ San 56-1, 
Shinrimdong\\ Gwanakgu \\Seoul 47907\\ Korea}
\email{hansol84@snu.ac.kr}
\author[Shin]{Hyung-Seok Shin}
\address{Department of Mathematical Sciences \\ Seoul National University\\ San 56-1, 
Shinrimdong\\ Gwanakgu \\Seoul 47907\\ Korea}
\email{shs83@snu.ac.kr}
\title{On Orbifold embeddings}
\begin{abstract}
For an orbifold, there is a notion of an orbifold embedding, which is more general than the one of sub-orbifolds. 
We develop several properties of orbifold embeddings. In the case of translation groupoids, we show that
such a notion is equivalent to a strong equivariant immersion.
\end{abstract}
\maketitle

\tableofcontents

\section{Introduction}
Orbifolds arise naturally in many areas such as topology, geometric group theory, symplectic geometry and so on. 
In the last decade, they have been actively studied after Chen and Ruan introduced a new cohomology
ring structure on orbifold cohomology \cite{CR}.  Orbifolds also naturally appears when there is a symmetry, such
as in symplectic reductions or in the presence of group actions. A very natural and basic question is to find the sub-objects for a given orbifold.
A suborbifold (a subset which is also an orbifold with the induced topology) turns out to be a very restrictive
notion. For example, given a product of two orbifolds the correct notion of diagonal $\Delta$(see Definition \ref{orbdiag}) does not
become a suborbifold but it is what is called an {\em orbifold embedding} into the product orbifold. Hence, it is clear that one should enlarge the class of sub-objects of an orbifold by including orbifold embeddings. Although
the definition of orbifold embeddings appear in \cite{ALR}, such a notion has not been studied further.
We became interested in this question of sub-objects in order to consider a proper notion of a Fukaya category of an orbifold.
In the Fukaya category of a symplectic manifold, its objects are Lagrangian submanifolds decorated with additional data, and
we believe that in the case of symplectic orbifolds, Lagrangian orbifold embeddings  should become an important object
in its Fukaya category. (In this paper, we do not consider the Fukaya category of an orbifold, which is left for future research.)

In this paper we give a slightly modified definition of orbifold embedding and  explore several properties of orbifold embeddings.  Also we prove that given an abelian orbifold embedding, the induced map between inertia orbifolds again becomes an orbifold embedding (Theorem \ref{thminertia}).

One drawback of the definition of orbifold embedding is that it is rather cumbersome to work with as it is defined using local data. Also we observe that this notion of orbifold embedding is not Morita invariant. The second half of this paper is to remedy this in the case of translation groupoids.
Namely, we construct an equivariant immersion in sections 7 and 8, from the data of an orbifold embedding to an orbifold groupoid which is Morita equivalent to a translation groupoid $[M/G]$.  Equivariant immersions are much easier to work with than
orbifold embeddings, hence this construction should be very useful in applications. Pronk and Scull \cite{PS} showed that for translation groupoids, Morita equivalence can be defined only using translation groupoids, and our result is also in a similar point of view. 
To construct such an equivariant immersion, we use the Hilsum-Scandalis map, which is reviewed in section 6.

Not all equivariant immersions give rise to orbifold embeddings, and we define what we call a {\em strong } equivariant
immersion, which is shown to give  an orbifold embedding. Also the equivariant immersion obtained  from orbifold
embeddings are also strong. Hence, in the case of translation groupoids, one can work with
strong equivariant immersions, instead of orbifold embeddings.  

\section*{Acknowledgments}
We would like to thank  anonymous referee for helpful comments in improving the exposition of the paper.
This work was supported by the National Research Foundation of Korea (NRF) grant funded by the Korea Government (MEST) (No. 2012R1A1A2003117)

\section{Orbifold groupoids}
In this section, we briefly recall well-known notions related to orbifold groupoids. We refer
readers to \cite{MP} or \cite{ALR} for details.
One can define orbifolds in terms of  local uniformizing charts (due to Satake).
\begin{definition}
An orbifold is a Hausdorff, second countable topological space $X$ with a collection of uniformizing charts $(V_\alpha, G_\alpha, \phi_\alpha : V_\alpha \to X)$ of $X$, where the finite group $G_\alpha$ acts effectively on the manifold $V_\alpha$, and continous maps $\phi_\alpha$ which descend to a homeomorphism $\overline{\phi_\alpha} $ of $V_\alpha / G_\alpha$  onto an open subset $U_\alpha \subset X$. This data is required to satisfy following conditions:
\begin{enumerate}
\item $\{U_\alpha \}$ is a covering of $X$.
\item (Local compatibility) For $x \in U_\alpha \cap U_\beta$, there exist an open neighborhood $U \subset U_\alpha \cap U_\beta$ of $x$ and a chart $(V, G, \phi : V \to X)$ of $U$ which embeds to $(V_\alpha, G_\alpha, \phi_\alpha)$ and $(V_\beta, G_\beta, \phi_\beta)$.
\end{enumerate}
\end{definition}

In the modern approach of orbifolds, one usually uses the language of groupoids in the definition of orbifolds. This generalizes the notion of classical orbifolds allowing noneffective orbifolds. Recall that a groupoid is a (small) category whose morphisms are all invertible. Giving a topological structure and smooth structure on groupoids, we get the notion of Lie groupoids. 

\begin{definition}
A topological groupoid $\CG$ is a pair of topological spaces $G_0 := Obj(\CG)$ and $G_1 := Mor(\CG)$ together with continuous structure maps:
\begin{enumerate}
\item The source and target map $s, t : G_1 \rightrightarrows G_0$, which assigns to each arrow $g \in G_1$ its source object and target object, respectively.
\item The multiplication map $m : G_1 {}_{s}\times_{t} G_1 \to G_1$, which compose two arrows.
\item The unit map $u : G_0 \to G_1$, which is a two-sided unit for the multiplication.
\item The inverse $i : G_1 \to G_1$, which assigns to each arrow its inverse arrow. This map is well-defined since all morphisms are invertible.
\end{enumerate}
If all of the above maps are smooth and $s$ (or $t$) is a surjective submersion(so that the domain $G_1 {}_s \times_t G_1$ of $m$ is a smooth manifold), then $\CG$ is called a Lie groupoid.
\end{definition}

\begin{definition}
Let $\CG$ be a Lie groupoid.
\begin{enumerate}
\item $\CG$ is proper if $(s, t) : G_1 \to G_0 \times G_0$ is a proper map.
\item $\CG$ is called a foliation groupoid if each isotropy group $G_x$ is discrete.
\item $\CG$ is \`etale if $s$ and $t$ are local diffeomorphisms.
\end{enumerate}
\end{definition}
Note the every \`etale groupoid is a foliation groupoid. It can be easily checked that a proper foliation groupoid $\CG$ has only finite isotropy groups $G_x := (s, t)^{-1}(x, x)$ for each $x \in G_0$

\begin{definition}
We define an orbifold groupoid to be a proper \`etale Lie groupoid.
\end{definition}

Let us recall morphisms and Morita equivalence of orbifolds.
\begin{definition}
Let $\CG$ and $\CH$ be Lie groupoids. A homomorphism $\phi : \CH \to \CG$ consists of two smooth maps $\phi_0 : H_0 \to G_0$ and $\phi_1 : H_1 \to G_1$, that together commute with all the structure maps for the two groupoids $\CG$ and $\CH$. It means that Lie groupoid morphisms are smooth functors between categories.
\end{definition}
The following notion of equivalence is restrictive (it does not define equivalence relation), and later we will recall Morita equivalence which is indeed the correct notion of equivalences between orbifold groupoids.
\begin{definition}\label{equivalence}
A homomorphism between $\phi : \mathcal{H} \to \mathcal{G}$ between Lie groupoids is called equivalence if 
\begin{enumerate}
\item[(i)] (essentially surjective)
the map 
$$t \pi_1 : G_1 \,_s \! \times_{\phi_0} H_0 \to G_0$$
defined on the fibered product of manifolds 
$$\{(g,y) \, | \, g \in G_1, y \in H_0, s(g) = \phi(y) \}$$ 
is a surjective submersion where $\pi_1 : G_1 \,_s \! \times_{\phi_0} H_0 \to G_1$ is the projections to the first factor;
\item[(ii)]the square
\begin{equation*}
\begin{diagram}
\node{ H_1} \arrow{s,r}{(s,t)}\arrow{e,t}{\phi_1} \node{G_1} \arrow{s,r}{(s,t)} \\
\node{ H_0 \times H_0}\arrow{e,t}{\phi_0 \times \phi_0} \node{G_0 \times G_0}
\end{diagram}
\end{equation*}
is a fibered product of manifolds.
\end{enumerate}
\end{definition}

An equivalence in the Definition \ref{equivalence}
may not have an inverse. The notion of Morita equivalence is obtained by formally inverting equivalences in Definition \ref{equivalence}.
 
\begin{definition}\label{def:Mor}
$\mathcal{G}$ and $\mathcal{G}'$ are said to be Morita equivalent if there exists a groupoid $\mathcal{H}$ and two equivalences
$$\mathcal{G} \stackrel{\phi}{\longleftarrow} \mathcal{H} \stackrel{\phi'}{\longrightarrow} \mathcal{G}.$$
\end{definition}
It is well known that ``Morita equivalence" defines an equivalence relation. (See the discussion below Definition 1.43 in \cite{ALR}.) It is clear from the definition that equivalence is a special case of Morita equivalence.  Lastly, we give an example of Morita equivalent groupoid which are not equivalent. For example, if $\mathcal{G}$ can be made by 
 tearing off some part $\mathcal{G}'$ and adding arrows which contains the original gluing information,
 then we have an equivalence from   $\mathcal{G}$ to  $\mathcal{G}'$.
However, since ``tearing off" process is not continuous, there is no map in the opposite direction in general.  
\begin{example}
Consider two orbifold groupoids,  $\mathcal{G}$ and $\mathcal{G}'$, which are equivalent to the closed interval. From the figure \ref{Moritaex}, it is clear that they are Morita equivalent, but there are no maps neither from $\mathcal{G}$ to $\mathcal{G}'$ nor from $\mathcal{G}'$ to $\mathcal{G}$.
\end{example}
\begin{figure}[h]
\begin{center}
\includegraphics[height=1.3in]{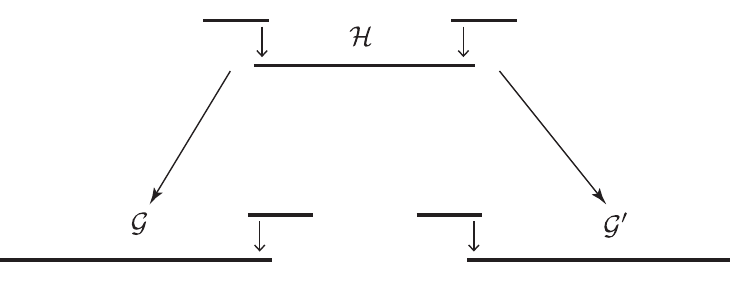}

\label{morita}
\caption{Morita equivalence}\label{Moritaex}
\end{center}
\end{figure}

\begin{definition}\label{fiberprod}
Let $\phi : \CH \to \CG$ and $\psi : \mathcal{K} \to \CG$ be homomorphisms of Lie groupoids. The fibered product $\CH \times_{\CG} \mathcal{K}$ is the Lie groupoid that makes the following diagram a fibered product.
\begin{equation*}
\xymatrix{ \CH \times_\CG \mathcal{K} \ar[dd]^{pr_1} \ar[rr]^{pr_2}&& \mathcal{K} \ar[dd]^{\psi}\\
\\
\CH \ar[rr]^{\phi} && \CG
}
\end{equation*}
which commutes up to a natural transformation. More explicitly,
\begin{eqnarray}\label{fiberform}
(\CH \times_{\CG} \mathcal{K})_0 &:=& H_0 \times_{\phi_0, G_0, s} G_1 \times_{t, G_0, \psi_0} K_0\\
(\CH \times_{\CG} \mathcal{K})_1 &:=& H_1 \times_{s\phi_1, G_0, s} G_1 \times_{t, G_0, s\psi_1} K_1
\end{eqnarray}
with following source and target maps
\begin{align*}
s(h, g, k) &= (s(h), g, s(k)),\\
t(h, g, k) &= (t(h), \psi(k)g\phi(h)^{-1}, t(k)).
\end{align*}
We will also write $\CH \times_\CG \mathcal{K}$ as $\phi^\ast \mathcal{K}$ occasionally.
\end{definition}

To be more precise, an element of $(\CH \times_{\CG} \mathcal{K})_0$ is a triple $(x,g,z)$ such that
\begin{equation}
\xymatrix{ x & \phi_0 (x) \ar[r]^g & \psi_0 (z) & z}
\end{equation}
and a morphism between $(x,g,z)$ and $(x',g',z')$ is a triple $(h,g,k)$ which makes the following diagram commutative.
\begin{equation}
\xymatrix{ x \ar[d]^h & \phi_0 (x) \ar[d]_{\phi_1 (h)} \ar[r]^g & \psi_0 (z) \ar[d]^{\psi_1 (k)} & z \ar[d]^k\\
x' & \phi_0 (x') \ar[r]^{g'} & \psi_0 (z') & z' }
\end{equation}
i.e. for $(h,g,k) \in (\CH \times_{\mathcal{K}} \CG)_1$ which satisfies $s(g) = \phi_0 (s(h))$ and $t(g) =\psi_0 (s(k))$ by definition,
\begin{equation*}
\xymatrix{ s(h,g,k)\,\,\,=\quad s(h) & \phi_0 (s(h)) \ar[r]^g & \psi_0 (s(k)) & s(k)}
\end{equation*}
and
\begin{equation*}
\xymatrix{ t(h,g,k)\,\,\,=\quad t(h) & \phi_0 (t(h)) \ar[r]^{g'} & \psi_0 (t(k)) & t(k)}
\end{equation*}
where $g'=\psi_1 (k) g \phi_1 (h)^{-1}$.

\begin{remark}\label{fiber}
The fibered product $\CH \times_{\CG} \mathcal{K}$ may not be a Lie groupoid, since $(\CH \times_{\CG} \mathcal{K})_0$ or $(\CH \times_{\CG} \mathcal{K})_1$ may not be manifolds.
\end{remark}
The following lemma is well-known.
\begin{lemma}\label{pbequiv}
If $\psi : \mathcal{K} \to \CG$ is an equivalence, then $\CH \times_{\CG} \mathcal{K}$ is a Lie groupoid and the projection $\CH \times_{\CG} \mathcal{K} \to \mathcal{H}$ is an equivalence
\begin{equation}\label{fibdia}
\xymatrix{
\CH \times_{\CG} \mathcal{K} \ar[rr]^{pr_2} \ar[dd]_{pr_1}& & \mathcal{K} \ar[dd]^{\psi : \cong} \\
&&\\
\mathcal{H} \ar[rr]_{\phi}&& \mathcal{G}. }
\end{equation}
\end{lemma}

\begin{proof}
From \eqref{fiberform}, one can see that if $s\circ pr_1 : G_1 \times_{t, G_0, \psi_0} K_0 \to G_0$ is a submersion, then $(\CH \times_{\CG} \mathcal{K})_0$ is a manifold. This happens when $\psi$ is an equivalence. Since $s:K_1 \to K_0$ is a submersion, a similar argument shows that $(\CH \times_{\CG} \mathcal{K})_1$ is a manifold for the equivalence $\psi$.

Recall that  $\CH \times_{\CG} \mathcal{K}$ is a Lie groupoid  whose set of objects and arrows are
\begin{eqnarray*}
(\CH \times_{\CG} \mathcal{K})_0 &=& H_0 \times_{\phi_0, G_0, s} G_1 \times_{t, G_0, \psi_0} K_0,\\
(\CH \times_{\CG} \mathcal{K})_1 &=& H_1 \times_{s\phi_1, G_0, s} G_1 \times_{t, G_0, s\psi_1} K_1
\end{eqnarray*}
respectively.
We first check the condition (i) of Definition \ref{equivalence}. We have to show that the following map
\begin{equation*}
t\pi_1 : H_1 \times_{s, H_0, pr_1} (H_0 \times_{\phi_0, G_0, s} G_1 \times_{t, G_0, \psi_0} K_0) \to H_0
\end{equation*}
is a surjective submersion where $\pi_1$ is the projection to the first factor $H_1$. Consider the following diagrams of fiber products.
\begin{equation*}
\xymatrix{H_1 \times_{s, H_0,pr_1} (\CH \times_{\CG} \mathcal{K})_0 \ar[r] \ar[d]_{\pi_1} &(\CH \times_{\CG} \mathcal{K})_0  \ar[r] \ar[d] & G_1 \times_{t, G_0, \psi_0} K_0 \ar[d] \\
H_1 \ar[r] & H_0 \ar[r] & G_0 \\
}
\end{equation*}
The rightmost vertical map $G_1 \times_{t, G_0, \psi_0} K_0 \to G_0$ is a surjective submersion, since $\psi : \mathcal{K} \to \mathcal{G}$ is an equivalence. Then, it follows from a general property of fiber product diagrams that the middle vertical map $H_0 \times_{\phi_0, G_0, s} G_1 \times_{t, G_0, \psi_0} K_0 \to H_0$ is also a surjective submersion, and hence so is $\pi_1$. Finally, $t \pi_1$ is a surjective submersion since it is given by a composition of two such kinds of maps.
 
To show the second condition of equivalence, we consider the following diagram
\begin{equation*}
\begin{diagram}
\node{(\CH \times_{\CG} \mathcal{K})_1} \arrow{s,r}{(s,t)} \arrow{e,t}{pr_1} \node{H_1} \arrow{s,r}{(s,t)} \\
\node{(\CH \times_{\CG} \mathcal{K})_0 \times (\CH \times_{\CG} \mathcal{K})_0}\arrow{e,t}{pr_1 \times pr_1} \node{H_0 \times H_0}
\end{diagram}
\end{equation*} 
Since $(s, t) : H_1 \to H_0 \times H_0$ is a submersion, we only need to check that the above diagram  is a fibered product of sets.
Suppose $h \in H_1$, and denote $x = s(h)$ and $x' = t(h)$.
Since $pr_1 : H_0 \times_{\phi_0, G_0, s} G_1 \times_{t, G_0, \psi_0} K_0 \to H_0$ is surjective, there exists $(x, g, y)$ and $(x', g', y')$ in $H_0 \times_{\phi_0, G_0, s} G_1 \times_{t, G_0, \psi_0} K_0$.
Since $\psi$ is equivalence, there exists a unique $k \in K_1$ satisfying $\psi_1 (k) = g' \phi_{1}(h) g^{-1}$. Since $h \in H_1$ determines a unique element $(h, g, k)$ in the fiber over $((x, g, y), (x', g', y'))$, the above diagram is a fiber product as sets.
\end{proof}

\section{Orbifold embeddings}
In this section we recall the main definition of this paper, the one of an orbifold embedding, and explore its properties. The following notion is a slight modification from the one defined by Adem, Leida, and Ruan in their book \cite{ALR}.
\begin{definition}\label{suborbifold}
A homomorphism of orbifold groupoids $\phi : \CH  \to \CG$ is an embedding if the
following conditions are satisfied:
\begin{enumerate}
\item $\phi_0 :H_0 \to G_0$ is an immersion
\item Let $x \in im(\phi_0) \subset G_0$ and let $U_x$ be a neighborhood such that
$\CG|_{U_x} \cong G_x \ltimes U_x$. Then, the $\CH$-action on $\phi_0^{-1}(x)$ is transitive, and there exists an open neighborhood
$V_y \subset H_0$ for each $y \in \phi^{-1}_0(x)$ such that $\CH|_{V_y} \cong H_y \ltimes V_y$
and 
\begin{equation}\label{model}
\CH|_{\phi^{-1}_0(U_x)} \cong G_x \ltimes (G_x  \times_{H_y} V_y)
\end{equation}
\item $|\phi|:|\CH| \to |\CG|$ is proper and injective.
\end{enumerate}

$\CH$ together with $\phi$ is called an orbifold embedding of $G$.
\end{definition}

\begin{figure}[h]
\begin{center}
\includegraphics[height=1.5in]{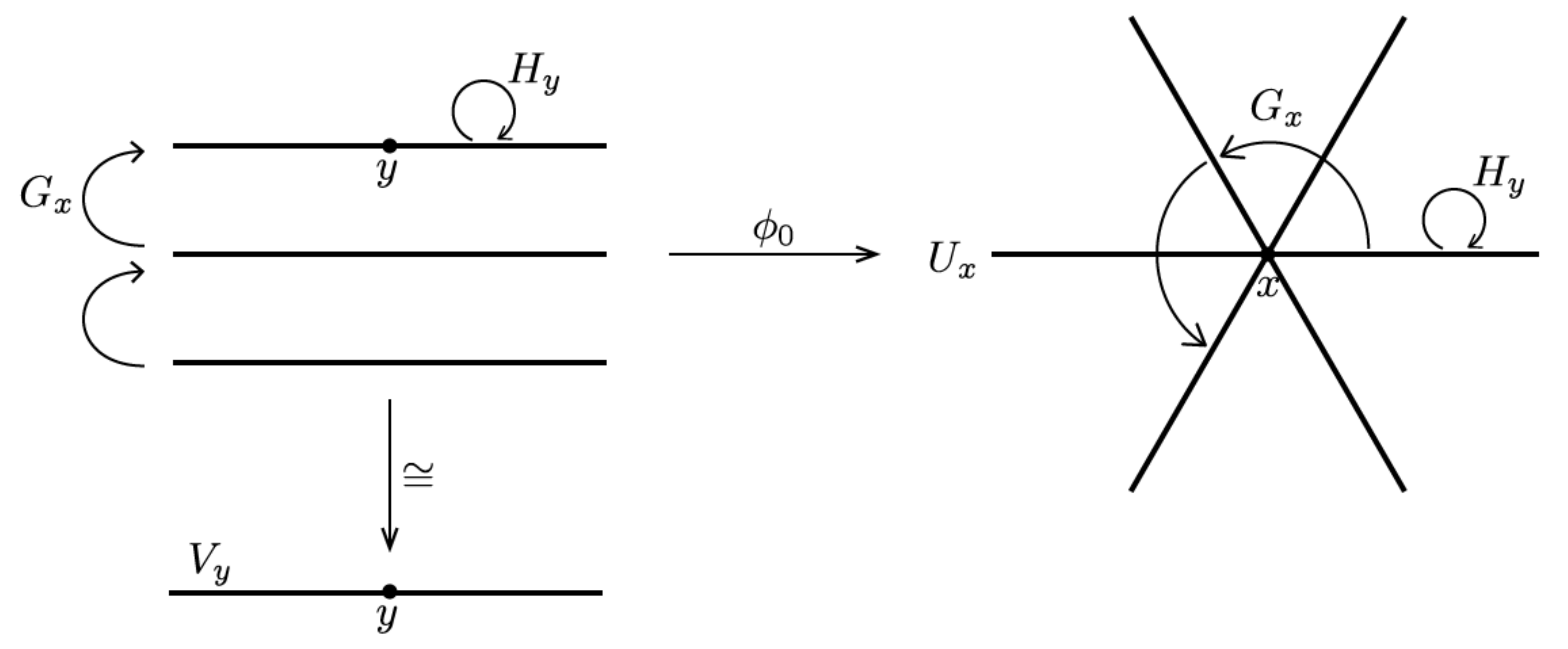}

\label{orbemb}
\caption{Local shape of an orbifold embedding}
\end{center}
\end{figure}
In (2) of Definition \ref{suborbifold} the action of $G_x$ is defined by
$$G_x \times (G_x \times_{H_y} V_y) \to G_x \times_{H_y} V_y,\quad (g,[k,z]) \mapsto [gk,z]$$
where $(k\phi_1(h), z) \sim (k, h\cdot z)$ is the equivalence relation defined by the action of $H_y$ and $[k,z]$ denotes a class in the quotient $G_x \times_{H_y} V_y$. \\

There are two modifications in the definition from that of  Adem, Leida, and Ruan (Definition 2.3 in \cite{ALR}).
\begin{enumerate}
\item We use the local model  $G_x \times_{H_y} V_y$ instead of  $G_x /H_y \times V_y$.
\item We require that $|\phi|:|\CH| \to |\CG|$  is injective (which was not present in \cite{ALR}).
\end{enumerate}
Let us explain why we have made such modifications.

Firstly, in \cite{ALR} $G_x /\phi_1(H_y) \times V_y$ was used instead of $G_x \times_{H_y} V_y$.
But $\phi_1(H_y)$ may not be a normal subgroup of $G_x$ (See the example \ref{ex1}). 
Also, it is not easy to find a natural $G_x$ action on $G_x / \phi_1 (H_y) \times V_y$ which reflects the $H_y$ action on $V_y$. The only plausible action of $G_x$ that may exist on $G_x /\phi_1(H_y)  \times V_y$ is by the left multiplication on the first component. Now, any reasonable definition of an  embedding should include the identity map, and therefore in this case we would have that $G_x \ltimes U_x \cong G_x \ltimes (G_x/G_x \times V_y)$ where $x=y$ and $U_x =V_y$ but, on $U_x$ the group $G_x$ acts and on $G_x / G_x \times V_y$ the action is trivial. Hence, $G_x \times_{H_y} V_y$  in \eqref{model} should be the correct local model.

\begin{example}\label{ex1}
Let $S_3$ act on $\CC^3$ as permutations on three coordinates where $S_3$ is the permutation group on $3$ letters. Consider $V:= \CC \times \CC \times \{0 \} \subset \CC^3$ and the subgroup $H$ of $S_3$ generated by the transposition $(1,2)$. Then, $H$ acts on $V$ and the natural map
$$ S_3 \times_H V \to \CC^3  $$
induces an orbifold embedding $S_3 \ltimes \left( S_3 \times_H V\right) \to S_3  \ltimes \CC^3$. Note that $H$ is not a normal subgroup of $S_3$.
\end{example}

Secondly, in \cite{ALR}, an orbifold embedding $\phi : \mathcal{H} \to \mathcal{G}$ does not necessarily induce an injective map $|\phi| : |\mathcal{H}| \to |\mathcal{G}|$. We first provide an example where $|\phi|$ is not injective but satisfies the other conditions of embedding. We will call a morphism $\phi$ of Lie groupoids {\em essentially injective} if $|\phi|$ is injective.

\begin{example}\label{nonessential}
Let $\mathcal{G}$ be given by $G_0 = \RR \amalg \RR$ and $\ZZ / 2\ZZ$ identifying two copies of $\RR$. Suppose $\mathcal{H}$ is the disjoint union of two copies of $\RR$ with only trivial arrows.

Immerse (embed) $H_0$ to $G_0$ by $id_\RR \amalg id_\RR$. One can easily check that $\phi$ satisfies the other axioms of orbifold embedding, but $|\phi|$ not injective. The induced map between quotient space is rather a covering map from trivial double cover of $\RR$ to $\RR$.
\end{example}



\begin{remark}
A morphism of groupoids $\phi : \mathcal{H} \to \mathcal{G}$ is essentially injective if the $\mathcal{H}$-action on $\phi_0^{-1} (t (s^{-1} (y) ) )$ ($\phi_0$ inverse image of $H_1$-orbit) is transitive for every $y \in G_0$, i.e. if there exists an arrow in $G_1$ from $\phi_0(x)$ to $\phi_0 (x')$, then one can find an arrow in $H_1$ from $x$ to $x'$. 

Compare it with the notion of essential surjectivity: $\phi : \mathcal{H} \to \mathcal{G}$ is called essentially surjective if for any point $x$ in $G_0$, there is an arrow $g : \phi(y) \to x$ from a point in the image of $\phi$ to $x$. 
\end{remark}





\begin{remark}
The essential injectivity is a property which is Morita-invariant since it is a property of the induced map between quotient spaces. 
\end{remark} 

Let us mention why such a notion of orbifold embedding is needed.
First, for orbifolds, one can define suborbifolds as sub Lie groupoids which are orbifold groupoids.
But important objects, such as the diagonal, do {\em not} become suborbifolds. Hence we need a
proper notion to consider such objects, or we need to enlarge the definition of suborbifolds to include orbifold embeddings. See example \ref{orbdiag} to note that the diagonal homomorphism 
is indeed an orbifold embedding. We believe that orbifold embeddings are an important
class of subobjects for an orbifold, but unfortunately, these notions has not been so far used nor
developed further. 

From now on, we develop the properties of orbifold embeddings.
\begin{lemma}\label{localinject}
If $\phi : \CH \to \CG$ is an orbifold embedding, then the restriction of $\phi_1$ on local isotropy groups is injective.
\end{lemma}
\begin{proof}
Note that the point $y$ corresponds to $[e,y]$ in this model, where $e$ is the identity element in $G_x$. Since equivalence between orbifolds preserves local isotropy groups, the local group $\phi_1 (H_y)$ at $[e, y]$ of $G_x \ltimes (G_x \times_{H_y} V_y)$ has to be isomorphic to $H_y$, and it proves the lemma.
\end{proof}

\begin{remark}
For the case of an effective orbifold $\CH$, Lemma \ref{localinject} follows directly from the 0-level immersion $\phi_0$. Assume that there is a nontrivial element $h \in Ker(\phi_1 |_{H_y})$. Fix a tangent vector $v \in T_{y} V_y$. Since the action of $\CH$ is effective, the difference of two vectors $v - h_{*}v$ is not a zero vector. By the assumption on $h$,
$$(\phi_0)_{*}(v - h_{*}v) = 0,$$
and it contradicts that $\phi_0$ is an immersion.
\end{remark}

\begin{remark}
In the above Lemma \ref{localinject}, $\phi_1$ may not be globally injective.
\end{remark}

We remark that the orbifold embedding is not  Morita invariant. Indeed, the following two examples illustrate this phenomenon.

\begin{figure}[h]
\begin{center}
\includegraphics[height=2.3in]{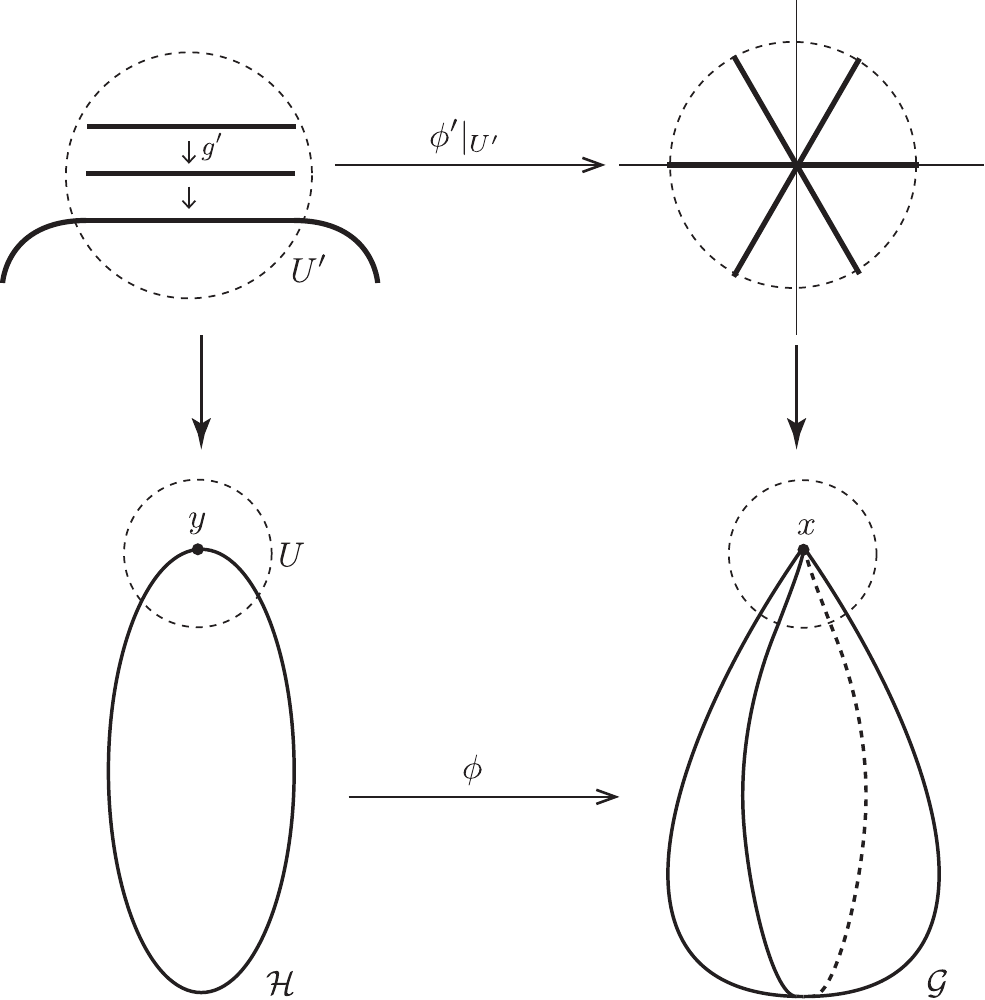}
\caption{An orbifold embedding and equivalence 1}\label{Morita}
\end{center}
\end{figure}

\begin{example}
Let $\mathcal{H}$ be a circle with the trivial orbifold groupoid structure and $\mathcal{G}$ be a teardrop whose local group at the unique singular point $x$ is $\ZZ /3$ as in Figure \ref{Morita}. The orbifold morphism $\phi : \mathcal{H} \to \mathcal{G}$ is not an orbifold embedding since it does not satisfy the second condition at $x \in \mathcal{G}$. 

However, we can change the orbifold structure of $\mathcal{H}$ as follows. Let $\phi(y) =x$ and $U$ be a open neighborhood of $y$ as in the figure. We add two more copies of $U$ to get new objects $U'$ and add additional arrows identifying three copies of $U$. Denote by $\mathcal{H}'$ the resulting orbifold. Note that there is an equivalence from $\mathcal{H}'$ to $\mathcal{H}$. The obvious modification $\phi' : \mathcal{H}' \to \mathcal{G}$ of $\phi$ is now an orbifold embedding.
\end{example}

\begin{figure}[h]
\begin{center}
\includegraphics[height=2.2in]{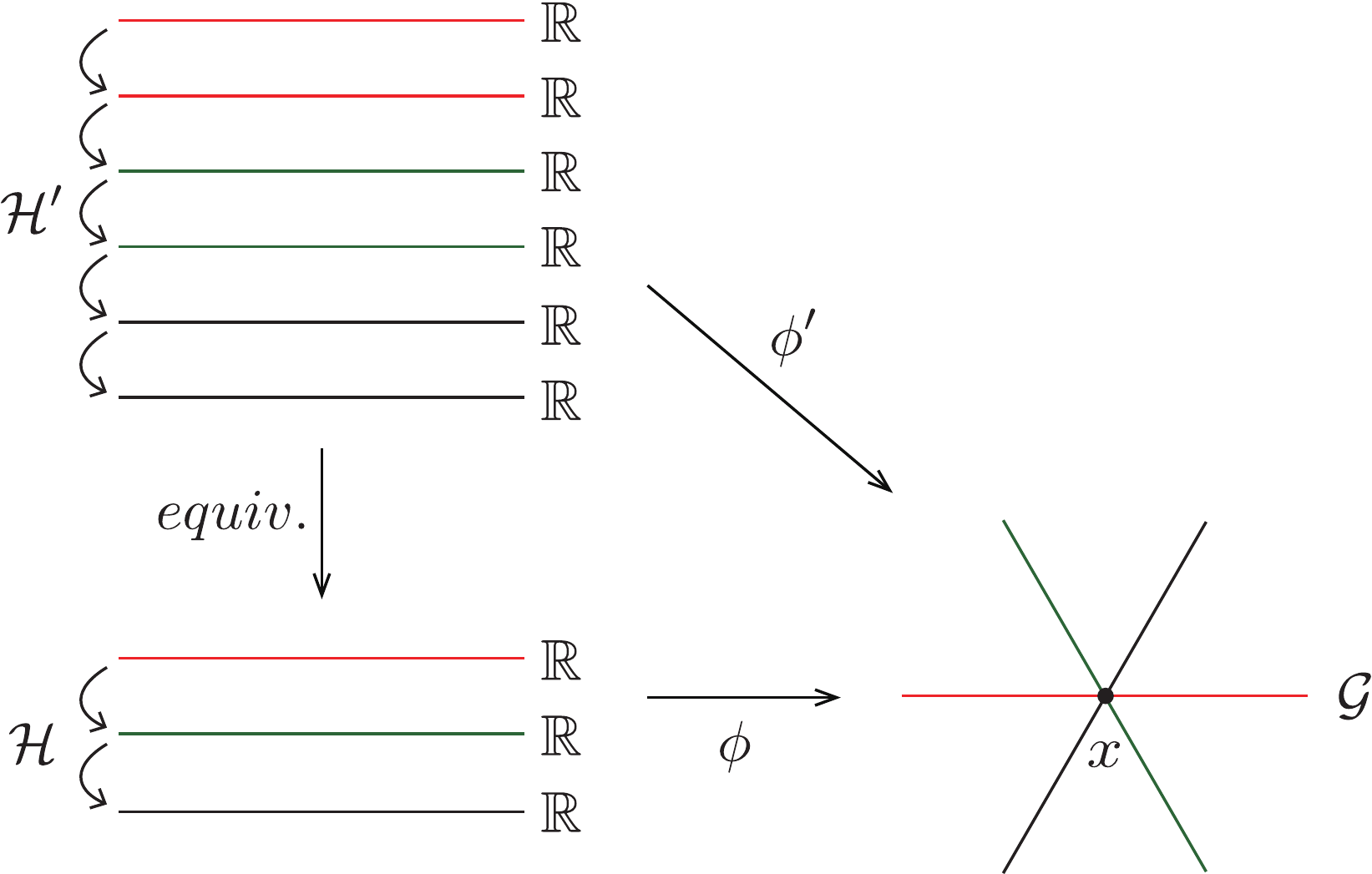}
\caption{An orbifold embedding and equivalence 2}\label{Morita2}
\end{center}
\end{figure}

\begin{example}
Let $\mathcal{H}$ be the disjoint union of three copies of real lines and $\mathcal{G}$ be $\RR^2$ equipped with a $\ZZ /3$ action generated by $2 \pi /3$-rotation. Consider an orbifold embedding $\phi : \mathcal{H} \to \mathcal{G}$ shown in Figure \ref{Morita2}. We similarly change the orbifold structure of $\mathcal{H}$ by adding three more copies of $\RR$ to $\mathcal{H}$ to get a new orbifold groupoid $\mathcal{H}'$, i.e. $H'_0 = \RR \times \ZZ_6$ and $((h,g),k) \in H'_1 = H'_0 \times \ZZ_6$ sends $(h,g) \to (h, kg)$. It is clear from Figure \ref{Morita2} that there is an equivalence $ \mathcal{H}' \to \mathcal{H}$, which is induced by the projection $\ZZ_6 \to \ZZ_3$. The morphism $\phi' : \mathcal{H}' \to \mathcal{G}$ is defined by the composition of $\phi$ and this equivalence. Then, we see that $\phi'$ is no longer an orbifold morphism because there is no transitive $G_x$ action on $\phi'^{-1} (x)$ where $x$ is the unique singular point in $\mathcal{G}$.
\end{example}

\begin{example}[{\bf Orbifold diagonal}]
As an example of an orbifold embedding, we introduce a diagonal suborbifold of product orbifolds.
\begin{definition}\label{orbdiag}
The diagonal suborbifold $\Delta$ is defined as $\CG \times_\CG \CG$ 
\end{definition}
\begin{lemma}
The natural map $\Delta=\CG \times_\CG \CG \to \CG \times \CG$ is an orbifold embedding.
\end{lemma}

\begin{proof}
We verify that $\Delta$ above is a subgroupoid in the sense of Definition \ref{suborbifold}. It suffices to prove this when $\CG$ is a global quotient orbifold $G \ltimes M$. In this case, $\Delta$ is given by $(G \times G) \ltimes (\sqcup_g \Delta_g )$, where $\Delta_g = \{ (x, gx) : x \in M\}$ and $(h,k) \in G \times G$ takes $(x, g, gx)$ to $(hx, kgh^{-1}, kgx)$. (The second terms in the triples are used to distinguish $(x, gx)$ from $(x, g h x)$ for $h \in G_x$.) The most natural choice of orbifold morphisms will be $\phi_0 :(x,g,gx) \mapsto (x,gx) \in M \times M$ and $\phi_1 : (h,k) \mapsto (h,k) \in G \times G$. $\phi_0$ is clearly an immersion. Note that $\phi_1$ is injective.

Choose a point $p=(x,y)$ in $im (\phi_0)$. Let $U$ be a small connected neighborhood of $p$ in $M \times M$, which is preserved under the $(G \times G)_p$-action. Since $(x,y)$ is in the image of $\phi_0$, there is some $g$ in  $G$ satisfying $y=gx$. So, in particular $\phi_0^{-1} (x,y) = \{ (x,g', g'x) \, | \, g'x = gx = y\}$. $(G\times G)_p$ acts on $\phi_0^{-1} (p)$ transitively since $(g^{-1} g',  e)$ sends $(x,g',g'x) \in \phi_0^{-1} (p)$ to $(x,g,gx)$ and $g^{-1} g' \in G_x$ when $g'x = gx$.

Let $q$ denote $(x,g,gx) \in \phi_0^{-1} (p)$. Note that $(G \times G)_{q} = \{ (h,k) : h \in G_x, k \in G_{gx}, kgh^{-1} = g\}$. Let $V_g$ be the connected component of $\phi_0^{-1} (U)$ which contains $q$ ($V_g$ is given by $\Delta_{g} \cap \phi_0^{-1} (U)$). We define a smooth map $\psi$ from $(G \times G)_p \times V_g$ to $\phi_0^{-1} (U)$ by
\begin{equation}\label{diaglocalmod}
\psi :  ((h,k),(x',g,gx')) \mapsto (hx', kgh^{-1}, kgx').
\end{equation}
Then, $\psi$ is $(G \times G)_p$-equivariant by the definition of the $G \times G$-action on $\Delta$. Since $U$ is preserved under the $(G \times G)_p$ and $V_g$ is a connected component, $\psi$ should be surjective. 

Suppose two different points 
$$q_1 = ((h_1,k_1),(x_1,g,gx_1)) \quad \mbox{and} \quad q_2= ((h_2,k_2),(x_2,g,g x_2))$$
in $(G \times G)_p \times V_g$ are mapped to the same point in $\phi_0^{-1} (U)$ by $\psi$. This happens precisely when $(h_2^{-1} h_1, k_2^{-1} k_1)$ sends $(x_1, g, g x_1)$ to $(x_2, g, g x_2)$. In particular, we have $(h_2^{-1} h_1, k_2^{-1} k_1) \in (G \times G)_{q}$. Therefore, $\psi$ descends to a map
$$ \bar{\psi} : (G \times G)_p \times_{(G \times G)_q} V_g \to \phi^{-1} (U)$$
which is bijective. (Here, $(a,b) \in (G \times G)_q$ acts on the first factor of $(G \times G)_p \times V_g$ by $(h,k) \mapsto (ha^{-1}, kb^{-1})$.) Since the $(G \times G)_q$-action and the $(G \times G)_p$-action on $(G \times G)_p \times V_g$ commute, the $(G \times G)_p$-equivariance of $\psi$ implies that of $\bar{\psi}$.
\end{proof}
\end{example}
\section{Inertia orbifolds and orbifold embeddings}
In this section we show that given an orbifold embedding, there is an induced orbifold embedding
between their inertia orbifolds under abelian assumption.

First, let us recall inertia orbifolds.
 The following diagram defines a smooth manifold $\mathcal{S}_{\mathcal{G}}$, which can be interpreted intuitively as a set of loops (i.e. elements of local groups) in $\mathcal{G}$:
\begin{equation}
\begin{diagram}
\node{  \mathcal{S}_{\mathcal{G}} }\arrow{s,r}{\beta}\arrow{e,t}{} \node{G_1} \arrow{s,r}{(s,t)} \\
\node{G_0}\arrow{e,t}{diag} \node{G_0 \times G_0.}
\end{diagram}
\end{equation}
Then, the inertia orbifold $\Lambda \mathcal{G}$ will be an action groupoid $\mathcal{G} \ltimes \mathcal{S}_{\mathcal{G}}$. i.e.
$$ (\Lambda \mathcal{G} )_0 = \mathcal{S}_{\mathcal{G}},$$
$$ (\Lambda \mathcal{G} )_1 = G_1 \times_{G_0} \mathcal{S}_{\mathcal{G}} $$
where for $h \in G_1$ the induced map $h : \beta^{-1} ( s(h) ) \to \beta^{-1} ( t(h) )$ is given by the conjugation. More precisely, for any $g \in \beta^{-1} ( s(h) )$, set $h(g) = h g h^{-1}$. This gives a target map from $   (\Lambda \mathcal{G} )_1  $ to $ (\Lambda \mathcal{G} )_0  $ whereas the source map is simply the projection to the second factor of $(\Lambda \mathcal{G} )_1$. Note that $\beta^{-1} ( s(h) ) $ and $\beta^{-1} ( t(h) )$ are the sets of loops in $\mathcal{G}$ based at $s(h)$ and $t(h)$, respectively. Similarly, one can define $\mathcal{S}_{\mathcal{H}}$ and $\Lambda \mathcal{H}$ for a suborbifold $\CH$ of $\CG$. \\

Now, let us see how $\phi : \mathcal{H} \to \mathcal{G}$ induces a morphism $\Lambda \phi$ between inertia orbifolds. $\Lambda \phi_0$ should be a map from $(\Lambda \mathcal{H})_0 = \mathcal{S}_{\mathcal{H}}$ to $(\Lambda \mathcal{G})_0  = \mathcal{S}_{\mathcal{G}}$. Suppose $(h, y)$, $y=\beta(h)\in H_0$ is a loop $h : y \to y $ in $\mathcal{H}$. Then, the image of this loop is $(\phi_1 (h), \phi_0 (y) )$ or, $\phi_1 (h) : \phi_0 (y) \to \phi_0 (y)$, i.e.
$$ \Lambda \phi_0 : (h, y) \mapsto (\phi_1 (h), \phi_0 (y)).$$
$\Lambda \phi_1$ maps $(h', h) \in (
\Lambda \mathcal{H})_1 = H_1 \times_{H_0} \mathcal{S}_{\mathcal{H}}$ as follows:
$$\Lambda \phi_1 : (h', h) \mapsto (\phi_1 (h'), \phi_1 (h) ).$$
If $h : y \to y$, then $\phi_1 (h) : \phi_0 (y) \to \phi_0 (y)$.

\begin{lemma}
If $\CG$ is abelian, i.e. $G_x$ is an abelian group for each $x \in \CG_0$,
then $\Lambda \CH$-action on $\Lambda \phi_0^{-1} (g,x)$ is transitive.
\end{lemma}

\begin{proof}
To observe the local behavior of $\Lambda \phi$, we use the local model of embeddings. Near $y \in H_0$, the local model and the morphism, again denoted by $\phi$, is given as follows:
$$ \phi : G_x \ltimes (G_x \times_{H_y} V_y) \to G_x \ltimes U_x ,$$
where $V_y$ and $U_x$ are suitable neighborhoods of $y$ and $x$, respectively and $x=\phi_0 (y)$. Note that $\phi_0 : G_x \times_{H_y} V_y \to  U_x$ is given as $\phi_0 [g, y'] = g \cdot \phi_0 (y')$ and $\phi_1=(id, \phi_0) : G_x \times (G_x \times_{H_y} V_y ) \to G_x \times U_x$. One can easily check that $\phi$ is well-defined.

Recall that we assumed $\phi_1$ to be injective and identify $H_y$ as a subgroup of $G_x$. We observe the fiber $\Lambda \phi_0^{-1} (g, x)$ for a loop $g : x \to x$ in $\mathcal{G}$ in these local models.

In our local model, any objects in $\Lambda \phi_0^{-1} (g,x)$ can be written as $(g, [g', y])$ for some $g' \in G_x$.
Suppose that $(g, [g_1, y])$ and $(g, [g_2, y])$ are distinct objects in $\Lambda \phi_0^{-1} (g,x)$.
Now we want to find $k \in \Lambda \CH_1$ which sends $(g, [g_1, y])$ to $(g, [g_2, y])$, i.e. $k$ such that $k \cdot (g, [g_1, y])  =  (g, [g_2, y])$ or, equvalently $(kgk^{-1}, [k g_1, y]) = (g, [g_2, y])$. This can be simply achieved by choosing $k = g_2 g_1^{-1}$.

%
%

\end{proof}

For general $\CG$, $\Lambda \CH$-action on $\Lambda \phi_0^{-1} (g,x)$ is not necessarily transitive. 
In the last paragraph of the proof of the lemma, the abelian assumption is crucial to find $k \in \Lambda \CH_1$ satisfying $(kgk^{-1}, [k g_1, y]) = (g, [g_2, y])$. If $G_x$ is not abelian, such $k$ may not exist. One may try with $k = g_2 g_1^{-1}$ which sends $[g_1,y]$ to $[g_2,y]$, but the loop $kgk^{-1}$ is different from $g$ if $k$ does not commute with $g$. 
See the following example. 

%
%
%
%

\begin{example}
Let $G$ be the subgroup of ${\rm SL}\, (2, \CC)$ generated by
\begin{equation}
a=\left(\begin{array}{cc} \rho & 0 \\0 & \rho^{-1} \end{array}\right) \quad b=
\left(\begin{array}{cc}0 & 1 \\ -1 & 0\end{array}\right)
\end{equation}
where $\rho = e^{ \pi i /3}$. $G$ is called the binary dihedral group of order $12$. Consider its fundamental representation on $\CC^2$. Relations on generators $a$ and $b$ are given by 
$$a^6 = b^4 = 1,\quad b a b^{-1} = a^{-1}, \quad a^3 = b^2.$$
Let $V$ be the first coordinate axis in $\CC^2$. Then, the subgroup $H$ of $G$ generated by $a$  acts on $V$. Now,
$$G \ltimes (G \times_H V)$$ 
gives rise to an orbifold embedding into $[\CC^2 / G]$ whose image is the union of two coordinate axes in $\CC^2$. Note that on the level of inertia, $(a, [e, 0])$ and $(a, [b,0])$ in $\Lambda \left( G \ltimes (G \times_H V) \right)$ are both mapped to $(a, (0,0))$ in $\Lambda [\CC^2 /G]$ by the induced map between inertias. 

We claim that there is no arrow between $(a, [e, 0])$ and $(a, [b,0])$ in the inertia  $\Lambda \left( G \ltimes (G \times_H V) \right)$ and therefore, the induced map is not an orbifold embedding. Such an arrow would first send $[e,0]$ to $[b,0]$ and hence, it would be of the form $b h$ for some $h \in H$. This arrow sends the loop $a$ at $[e,0]$ to the loop $(bh) \, a \, (bh)^{-1}$ at $[b,0]$. However, for any $h \in H$, $(bh) \, a \, (bh)^{-1} = b a b^{-1} = a^{-1}$ since $H$ is abelian. 
\end{example}

Finally we prove that $\Lambda \phi : \Lambda\CH \to \Lambda \CG$ satisfies the condition (2) of the orbifold embedding \eqref{model} under the abelian assumption.

\begin{prop}\label{thminertia}
Given an  orbifold embedding $\phi: \CH \to \CG$, consider the induced map between inertia orbifolds $\Lambda \phi : \Lambda\CH \to \Lambda \CG$.
If $\CG$ is an abelian orbifold, then $\Lambda \phi$ is again an orbifold embedding.
\end{prop}
\begin{proof}
First of all, it is clear that $\Lambda \phi_0$ is an immersion. This follows from the fact that the sector fixed by the loop $h \in H_1$ should be mapped through $\phi_0$ to the sector fixed by $\phi_1 (h) \in G_1$. So ${\Lambda \phi_0}_\ast$ is essentially the same as  ${\phi_1}_\ast$, which sends tangent vector to the sector fixed by $h \in H_1$ to the one determined by $\phi_1 (h)$.

The only non-trivial part is the second condition. For this, we can work on local charts.
Suppose $\phi_0 (y) = x$ for $y \in \CH_0, x \in \CG_0$, and we fix a neighborhood $U_x$ of $x$ in $G_0$
from the embedding property of $\phi$, so that $H|_{\phi^{-1}(U_x)}$ can be identified with the
action groupoid $G_x \ltimes (G_x \times_{H_y} V_y)$.
We may identify $H_y$ as a subgroup of $G_x$ via the embedding map.

We fix an element $g \in H_y \subset G_x$. In general, the local chart of inertia orbifold $\Lambda \CG$ near $(g, x) \in \Lambda\CG_0$  can be written as 
\begin{equation}\label{inerchart}
C_G (g) \ltimes \big( U_x^g \times \{g\} \big),
\end{equation}
where $U_x^g$ is the set of $g$-fixed points in $U_x$ and $C_G (g) = \{ h \in G_x | hg = gh\}$ acts on $U_x^g$ by the left multiplication. We put $\{g\}$ in \eqref{inerchart} to indicate the sector in the inertia orbifold $\Lambda\CG$, and we will drop it for notational simplicity in the following.

In our case, $C_G (g) = G_x$ since $\CG$ is abelian. We rewrite the local chart \eqref{inerchart} as
\begin{equation}\label{inerchart2}
G_x \ltimes U_x^g.
\end{equation}
Choose an element $(g, [e, y])$ in $\Lambda \phi_{0}^{-1}(g, x)$. Then $V_y^g$ is an open neighborhood of $(g, [e, y])$ and 
$$\Lambda\CH|_{V_y^g} \cong H_y \ltimes V_y^g.$$
Note that $H_y = C_H (g) := \{ h \in H_y | hg = gh \}$, since $H_y$ is a subgroup of the abelian group $G_x$.

We claim that the inverse image of the local chart \eqref{inerchart2} of $(g,x)$ by $\Lambda \phi^{-1}$ can be written as
$$ G_x \ltimes \big(G_x \times_{H_y} V_y^g \big).$$
To see this, we only need to show that the twisted $g$-sector in the local chart of $\CH |_{\phi^{-1}(G_x \ltimes U_x)}$ is isomorphic to $G_x \ltimes \big(G_x \times_{H_y} V_y^g \big)$. It can be checked as follows:

If $[g', z] \in G_x \times_{H_y} V_y$ is a fixed point of $g$, then $[gg', z] = [g', z]$. By definition, this happens if $gg' = g'h$ and $hz = z$ for some $h \in H_y$. This is equivalent to $(g')^{-1}gg' z = z$, and by the abelian assumption on $G_x$, $gz = z$.
Hence, objects in $g$-twist sector of $\CH |_{\phi^{-1}(G_x \ltimes U_x)}$ are contained in $G_x \times_{H_y} V_y^g$. Conversely, using the condition that $G_x$ is abelian and $g \in H_y$, it follows that any element $[g', z] \in G_x \times_{H_y} V_y^g$ is fixed by $g$.

By the definition of arrows in an inertia groupoid,
$G_x \times ( G_x \times_{H_y} V_y^g )$ is the arrow space of the $g$-twisted sector of $\CH |_{\phi^{-1}(G_x \ltimes U_x)}$ with an obvious action map, and this proves the proposition.


%
%
\end{proof}

\section{Orbifold embeddings and equivariant immersions}
In this section we show that equivariant immersions which are {\em strong} (in the sense that will be defined later) give rise to orbifold embeddings between orbifold quotients.

First, let us review orbifold quotients and its groupoid analogue, translation groupoids.
Let $G$ be a compact Lie group which acts on $M$ smoothly. 
The quotient $[M/G]$ naturally has a structure of a translation groupoid.
\begin{definition}
Suppose a Lie group $G$ acts smoothly on a manifold $M$ from the left. The translation groupoid $[G \ltimes M]$ assoicated to this group action is defined as follows.
Let $(G \ltimes M)_0 := M$ and $(G \ltimes M)_1 := G \times M$, with $s : G \times M \to M$ the projection and $t : G \times M \to M$ the action. The other structure maps are defined in the natural way.
\end{definition}

In particular, we are interested in group actions which give rise to an orbifold groupoid structure.
\begin{definition}
A $G$-action on $M$ is said to be locally free if the isotropy groups $G_p$ are discrete for all $p \in M$.
\end{definition}
Now we assume that the $G$-action on $M$ is locally free. The compactness of $G$ implies that $G_x$ is finite for all $x \in M$. Since $G$ acts on $M$ locally freely, we have a representation of $[M/G]$ as an orbifold groupoids in the following manner which is called  the slice representation in \cite{MP}.

\begin{prop}
For any translation groupoid $[M/G]$, there is an orbifold groupoid $\CG$ with an equivalence groupoid homomorphism $p : \CG \to [M/G]$.
\end{prop}
\begin{proof}
By the slice theorem, we can cover $M$ by a collection of $G$-invariant open sets $\{U_i\}$ with $G$-equivariant diffeomorphisms
\begin{align*}
	\psi_i : G \times_{G_i} V_i \to U_i
\end{align*}
where $V_i$ is a normal slice with local action of $G_i \leq G$.
Define $\CG$ as follows. Let $G_0 := \sqcup_i V_i$ be the disjoint union of all the $V_i$, and define a map
$p : G_0 \to M$ as $p(i,v) := \psi_i([1,v])$.
Define $G_1$ as the pullback bundle of following diagram.
\begin{equation*}
\begin{diagram}
\node{G_1} \arrow{s,l}{(s,t)} \arrow{e} \node{G \times M} \arrow{s,r}{(s,t)}\\
\node{G_0 \times G_0} \arrow{e,b}{(p,p)} \node{M \times M}
\end{diagram}
\end{equation*}
Then groupoid homomorphism $p : \CG \to [M/G]$ is an equivalence. See the proof of Theorem 4.1 in \cite{MP} for more details.
\end{proof}

The converse in general still remains as a conjecture. The conjecture was partially proven in the case of effective orbifold groupoids (Theorem 1.23 of \cite{ALR}).
\begin{conjecture}\label{globalquot}
Every orbifold groupoid can be represented by translation groupoid with locally free group action.
\end{conjecture}

Now, let us recall the definition of an equivariant immersion and introduce what we call strong equivariant immersion.
\begin{definition}
Let $N, M$ be G-manifolds. A $G$-equivariant immersion from $N$ into $M$ is a smooth map $\iota : N \to M$ such that
\begin{enumerate}
\item the derivative $d \iota : T_x N \to T_{\iota(x)} M$ is injective at every point in $N$;
\item $\iota(g \cdot x)=g \cdot \iota(x)$.
\end{enumerate}
\end{definition}

When $\iota$ is an equivariant immersion,  the inverse image of $p \in \iota(N) \subset M$ admits a natural $G_p$ action. If $q \in N$ is a point in $\iota^{-1} (p)$, then for $g \in G_p$
\begin{equation}
\iota(g \cdot q) = g \cdot \iota(q) = g \cdot p = p.
\end{equation}

\begin{definition}
Suppose the $G$-action on $N$ is locally free and $\iota : N \to M$ be a $G$-equivariant immersion. We call $\iota$ a {\em strong} $G$-equivariant immersion if for every $p \in M$,  $G_p$ action on $\iota^{-1} (p)$ is transitive.
\end{definition}

Here is an example. Let $N$ be a submanifold of $M$, which may not be necessarily preserved by $G$-action. We take $G$ copies of $N$, and denote it by $\WT{N}$. i.e. $\WT{N} = G \times N$. $\WT{N}$ admits a natural $G$-action
\begin{equation}\label{WTL}
g : (h,x) \mapsto (gh,x)
\end{equation}
for $g,h \in G$ and $x \in N$.
An immersion $ \iota: \WT{N} \to M$ defined by $\iota (g,x) = g \cdot x$ is then $G$-equivariant.
\begin{lemma}

The $G$-equivariant immersion $\WT{\iota}:\WT{N} \to M$ obtained above is strong if and only if 
\begin{equation}\label{eq:strong}
N \cap g \cdot N = N^g. 						
\end{equation}
for all $g \in G$.
\end{lemma}

\begin{proof}
From the definition of $\iota$, only the image under $\iota$ of a point in $h \cdot N \cap g \cdot N$ can have a multiple fiber. Up to the $G$-action, it suffices to consider a point, say  $y \in N \cap g \cdot N$. Then there exists $x \in N$ such that $g \cdot x =y$. Observe that $(1,y)$ and $(g,x)$ in $\WT{N}$ maps to the same point $y \in M$.  For $\iota$ to be strong, there should be a group element mapping $(1,y)$ to $(g,x)$, and from \eqref{WTL}, this implies $x=y$. Therefore, $g \cdot x = x$ and, hence $x \in N^g$.
\end{proof}

When a nontrivial subgroup $G_N$ of $G$ preserves $N$ but do not fix $N$, then condition \eqref{eq:strong} cannot be satisfied in general. However, we may try to use the minimal number of copies of $N$. Define $G_N$ so that we have the property, $g \cdot N = h \cdot N$ if and only if $g^{-1} h \in G_N$. Thus, for an element $\alpha$ of the coset space $G/ G_N$, $\alpha N$ is well defined. Let
$$\WT{N} = \bigcup_{\alpha \in G/G_N} \alpha N \times \{ \alpha \}.$$
$\WT{N}$ is a $G$-space by letting 
$$g : (x, \alpha)  \mapsto (g \cdot x, g \cdot \alpha).$$
Obviously, the natural immersion $\iota : \WT{N} \to M$, $\iota( x, \alpha) = x$, is $G$-equivariant. 

With this construction, we can interpret the orbifold diagonal for a global quotient orbifolds (cf. \ref{orbdiag}) as a strong equivariant immersion: Suppose a finite group $G$ acts on $M$ and let $N$ be the diagonal submanifold of $M \times M$. Then, $G \times G / \Delta_G$ parametrizes sheets of the domain of the immersion where $\Delta_G = \{ (g,g) | g \in G \}$. i.e.
$$ \WT{N} = \bigcup_{\alpha \in G \times G / \Delta_G} \alpha N \times \{ \alpha \}.$$
To see that $\WT{N} \to M \times M$ is strong, assume $( x,gx, [1,g])$ and $(x,hx, [1,h])$ are mapped to the same point $z$ in $M \times M$ ($[1,g],[1,h] \in G \times G / \Delta_G$). Then, $(h^{-1} g, 1)$ belongs to the local isotropy of $z=(x,gx) = (x,hx) \in M \times M$ and it sends $( x,gx, [1,g])$ to $(x,hx, [1,h])$ since $h^{-1} g x = x $ and 
$$(h^{-1}g , 1) [1,g] = [h^{-1}g, g] = [h^{-1} ,1] = [1,h].$$

One nice property which follows from the strong condition is that the strong equivariant immersions always induce injective maps between the quotient spaces.
\begin{lemma}\label{underiota}
If $\iota : N \to M$ is a strong $G$-equvariant immersion between two $G$-spaces $N$ and $M$, then,
$$|\iota| : |N/G| \to |M/G|$$
is injective.
\end{lemma}

\begin{proof}
Let $|\iota| (\overline{q_1}) =|\iota| (\overline{q_2})$ in $|M/G|$ for $\overline{q_i} \in |N/G|$. Then,
$$\iota (q_1) = g \cdot \iota (q_2)$$
for some $g \in G$. Denote $\iota (q_1)$ by $p$. We have to find $h \in G$ such that $h \cdot q_1 = q_2$. Observe that
$$\iota (g \cdot q_2) = g \cdot \iota (q_2) = \iota (q_1) =p,$$
which implies that $g \cdot q_2$ and $q_1$ lie over the same fiber $\iota^{-1} (p)$ of $\iota$. Since $\iota$ is strong, there is $h' \in G_p$ such that $h' \cdot q_1 = g \cdot q_2$. By letting $h=g^{-1} h'$, we prove the claim.
\end{proof}

Next, we use the local model of strong $G$-equivariant immersion to construct an orbifold embedding. 
\begin{prop}\label{prop:eqem}
Let $\iota : N \to M$ be a strong $G$-equivariant immersion between two $G$-manifolds with locally free $G$-actions. Then, there exist orbifold groupoid representations $\mathcal{H}$ and $\mathcal{G}$ of $[N/G]$ and $[M/G]$ respectively so that $\iota$ induces an orbifold embedding $\phi_{\iota} : \mathcal{H} \to \mathcal{G}$ whose underlying map between quotient spaces is injective.
\end{prop}
We will give a proof at the end of this section, after we discuss local models.
The following lemma is  an analogue of standard slice theorem.
\begin{lemma}\label{Slice theorem}
Let $M$ be a manifold on which a compact Lie group $G$ acts locally freely. Suppose $\iota : N \to M$ is a $G$-equivariant immersion. For $q \in N$ and $p=\iota(q) \in M$, we can find a $G$-invariant neighborhood $\WT{U}_p$ of $G \cdot p$ in $M$ and $\WT{V}_q$ of $G \cdot q$ in $N$ with the following properties:
\begin{enumerate}
\item[(i)] There are normal slices $U_p$ and $V_q$ to $G \cdot p$ and $G \cdot q$ at $p$ and $q$ respectively such that
\begin{equation}\label{slice}
\WT{U}_p \cong G \times_{G_p} U_p \quad \WT{V}_q \cong G \times_{G_q} V_q.
\end{equation}
\item[(ii)] There is an $G_q$-equivariant embedding $e : V_q \to U_p$ such that the diagram
\begin{equation*}
\xymatrix{
\WT{V}_q \ar[r]^{\iota} \ar[d]_{\cong}& \WT{U}_p \ar[d]^{\cong}\\
 G \times_{G_q} V_q \ar[r]^{[id,e]} & G \times_{G_p} U_p
}
\end{equation*}
commutes where the map on the second row is given by $(g,v) \mapsto (g, e(v))$.
\end{enumerate}
\end{lemma}
\begin{proof}
This is a relative version of the slice theorem (see for example Theorem B.24 of \cite{GGK}). We briefly sketch the construction of the slice here. Fix any $G$-invariant metric $\xi$ on $M$. The exponential map $E$ identifies a neighborhood of $p$ in $M$ with a neighborhood of $0$ in $T_p M$. Moreover, $\xi$ induces a decomposition
\begin{equation}\label{nordecomp}
T_p M \cong T_p (G \cdot p ) \oplus W
\end{equation}
where $W$ is normal to the orbit and hence, it is equipped with the linear $G_p$ action (coming from the one on $T_p M / T_p (G \cdot p)$. Let $U_p \subset W$ be a $G_p$-invariant small disk in $W$ around the origin on which $E$ is a diffeomorphism. Now,
$$\psi : G \times_{G_p} U_p \to M, \quad [g,u] \mapsto g \cdot E(u)$$
is well defined and $G$-equivariant. Since $\psi$ is a local diffeomorphism at the point $[e,o]$, $G$-equivariance implies that it is a local diffeomorphism at all points of the form $[g,0]$. One can check that $\psi$ is indeed injective if $U_p$ is sufficiently small. (See the proof of Theorem B.24 in \cite{GGK} for details.)

To get the relative version, we pull back $\xi$ to $N$ by $\iota$. Since $\iota$ is an immersion, $\iota^\ast \xi$ gives a metric on $N$. From the $G$-equivariant injection $T_q N \stackrel{\iota_\ast}{\longrightarrow} T_p M$, we can choose a decomposition compatible with \eqref{nordecomp}:
$$T_q N \cong T_p (G \cdot q) \oplus W'$$
 i.e. $\iota_\ast$ is decomposed as
$$\iota_\ast=(\iota_\ast^O, \iota_\ast^N) :T_q (G \cdot q ) \oplus W' \to T_p (G \cdot p) \oplus W.$$
Note that $G_q \subset G_p$ and $\iota_\ast^N$ is $G_q$-equvariant. Let $V_q$ be the inverse image of $U_p$ by $\iota_\ast^N$. We may assume that $G \times_{G_q} V_q$ is diffeomorphic to a neighborhood of $G \cdot q$ by shrinking $U_p$ if necessary. Thus, we proved (i).

Finally, by letting $e$ the restriction of $\iota_\ast^N$ to $V_q$, we get (ii).
\end{proof}

\begin{figure}[h]
\begin{center}
\includegraphics[height=1.7in]{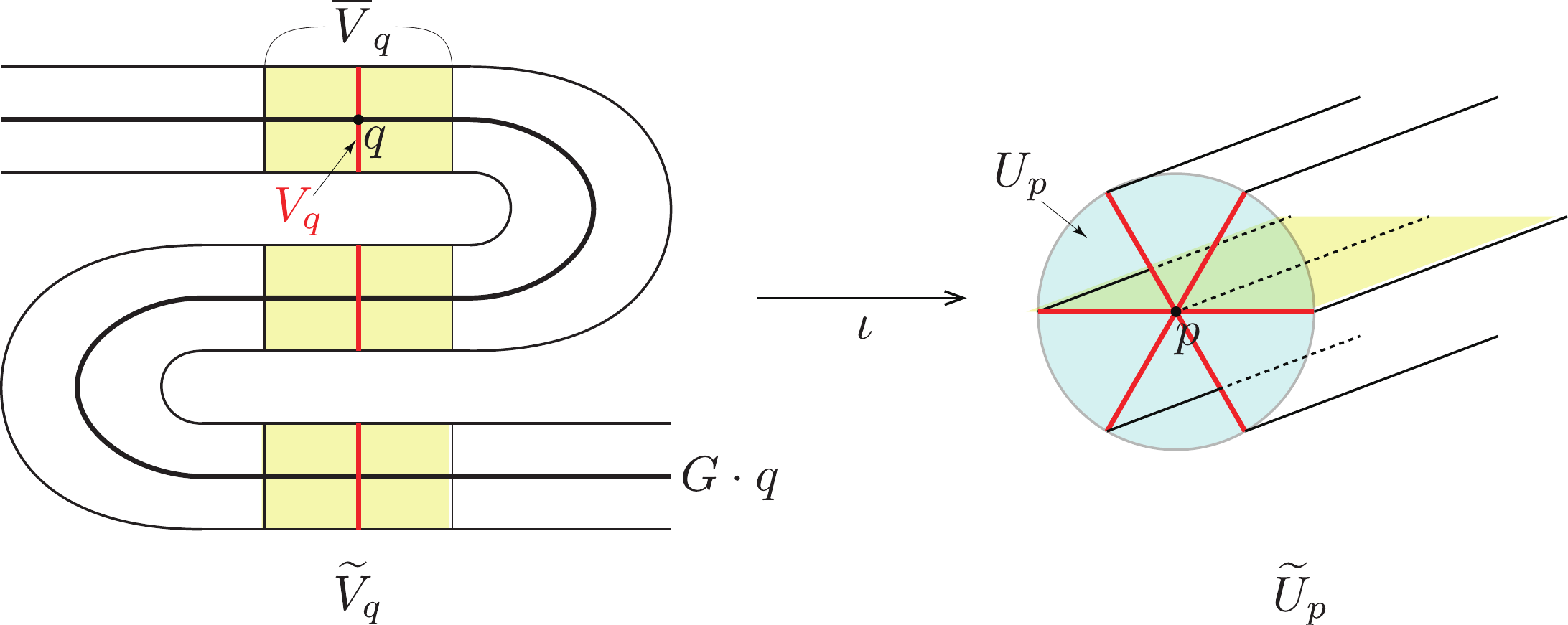}

\label{morita}
\caption{Slices}\label{stimmer}
\end{center}
\end{figure}

The following lemma provides the local model which is needed for the orbifold embedding.
\begin{lemma}\label{locmodelstrong}
Under the setting of Lemma \ref{Slice theorem}, assume further that $\iota$ is strong. Then, we can find a $G_p$-invariant neighborhood $\overline{U}_p(\subset \WT{U}_p)$ of $p$ in $M$ and a $G_q$-invariant neighborhood $\overline{V}_q (\subset \WT{V}_q)$ of $q$ in $N$ such that there is a $G_p$-equivariant isomorphism
\begin{equation}\label{simmer}
\iota^{-1} (\overline{U}_p) \cong G_p \times_{G_q} \overline{V}_q
\end{equation}
where the $G_q$-action on $G_p \times \overline{V}_q$ is given by
$$h \cdot (g,x) = (g h^{-1}, h \cdot x).$$
\end{lemma}
\begin{proof}
We take a {\it product type} neighborhood $\overline{V}_q(\subset \WT{V}_q)$ and $\overline{U}_p (\subset \WT{U}_p)$ of $q$ and $p$ as follows. We first identify the orbit $G \cdot q$ with $G / G_q$ where $q$ corresponds to the image of identity $[e]$ in $G /G_q$ and $G \cdot p$ with $G / G_p$ in a similar way. Then the tubular neighborhood $G \times_{G_q} V_q$ can be regarded as a fiber bundle over $G / G_q$. Take an open neighborhood $O_q$ of $[e]$ in $G / G_q$ which is invariant under the left $G_q$-action on $G / G_q$. As $[e]$ is fixed by this $G_q$-action, one can for instance choose left $G_q$-invariant metric on $G / G_q$ and then, take $O$ to be a small open ball around $[e]$. We may assume $O$ is small enough so that 
\begin{equation}\label{propdisc}
g \cdot O_q \cap O_q = \phi
\end{equation}
for nontrivial $g \in G_p \setminus G_q$. (Note that $G_p$ also acts on $G/ G_q$ from the left.) This is possible since $G_p$ is finite. Let $O_p$ be the image of $O_q$ by the map $G / G_q \to G / G_p$. \eqref{propdisc} implies that the map $O_q \to O_p$ is an embedding. 

Finally, we define $\overline{V}_q$ and $\overline{U}_p$ to be open neighborhoods of $q$ and $p$, respectively, such that the following diagrams are cartesian. (See Figure \ref{stimmer}.)
\begin{equation}
\xymatrix{ \overline{V}_q \ar[d] \ar@{^{(}->}[r] & G \times_{G_q} V_q \ar[d] & \overline{U}_p \ar[d] \ar@{^{(}->}[r] & G \times_{G_p} V_p \ar[d]  \\ 
O_q \ar@{^{(}->}[r] & G / G_q & O_p \ar@{^{(}->}[r] & G / G_p 
}
\end{equation}
Then, by \eqref{propdisc} we have
\begin{equation}\label{totdisc}
g \cdot \overline{V}_q =\left\{
\begin{array}{lc}
\overline{V}_q & \mbox{if}\,\,\, g \in G_q \\
\mbox{disjoint from}\,\, \overline{V}_q & \mbox{if}\,\,\, g \in G_p \setminus G_q
\end{array}\right.
\end{equation}
(More precisely, $\overline{V}_q$ and $\overline{U}_p$ are image of these fiber products under the isomorphisms shown in (i) of the previous lemma.)
Observe that $\iota|_{\overline{V}_q} : \overline{V}_q \to \overline{U}_p$ is an embedding since both $O_q \to O_p$ and $V_q \to U_p$ are embeddings. 

Since the $G_p$-action on $\iota^{-1} (p)$ is transitive, there is $|G_p| /|G_q|$-open subsets of $N$ (isomorphic to $\overline{V}_q$) which are mapped to $\overline{U}_p$. By \eqref{totdisc}, $\iota^{-1} (\overline{U}_p)$ is the disjoint union of these open subsets of $N$.

Now, define $\tilde{\phi} : G_p \times \overline{V}_q \to \iota^{-1} (\overline{U}_p)$ by
\begin{equation*}
\begin{array}{ccc}
\tilde{\phi} : G_p \times \overline{V}_q &\longrightarrow& \iota^{-1} (\overline{U}_p)  \\
(g,x) &\longmapsto& g \cdot x 
\end{array}
\end{equation*}
This map is well defined because $\iota$ is $G_p$-equivariant and $\overline{U}_p$ is $G_p$-invariant subset of $M$. Furthermore, $\tilde{\phi}$ is surjective (and hence a submersion) by the strong condition of $\iota$. It remains to show that it is injective up to $G_q$-action.

Suppose $\tilde{\phi}$ sends $(g,x)$ and $(g',x')$ in $G_p \times \overline{V}_q$ to the same point in $\iota^{-1} (\overline{U}_p)$. Then $g \cdot x = g' \cdot x'$, equivalently $\left(g'\right)^{-1}g \cdot x = x'$. Note that both $x$ and $x'$ belong to $\overline{V}_q$ and $\left(g'\right)^{-1}g \in G_p$. From the dichotomy \eqref{totdisc}, we have $\left(g'\right)^{-1}g =h$ for some $h \in G_q$. Therefore, $g' = g h^{-1}$ and $x' = h\cdot x$ for $h \in G_q$. We conclude that $\tilde{\phi}$ is indeed a principal $G_q$-bundle and the isomorphism \eqref{simmer} follows.
\end{proof}

\begin{remark}
Note that the induced map $\phi : G_p \times_{G_q} \overline{V}_q \to \iota^{-1} (\overline{U}_p)$ is $G_p$-equivariant by definition.
\end{remark}

\begin{proof}[Proof of proposition {\rm \ref{prop:eqem}.}]
Suppose we have a strong $G$-equivariant immersion $\iota : N \to M$. By equivariance, $\iota$ induces a map $\phi_{\iota}' : [N/G] \to [M/G]$ and $|\phi_{\iota}'| : N/G \to M/G$ is clearly injective. 

Consider a point $\bar{p}$ in $|M/G|$ and let $\pi_M (p) = \bar{p}$ for the quotient map $\pi_M : M \to |M/G|$. From the definition of strong equivariant immersion, the group action on $\iota^{-1} (p)$ is transitive. If $q \in N$ maps to $p$, then there exists a $G_p$-invariant product type  neighborhood $\overline{U}_p$ of $p$ in $M$ and $G_q$-invariant $\overline{V}_q$ of $q$ in $N$ which satisfies \eqref{simmer} from the lemma \ref{locmodelstrong}. We add all such slices $U_p$ and $V_q$ into the slice representations of $[N/G]$ and $[M/G]$ to get orbifold groupoids $\mathcal{H}$ and $\mathcal{G}$.

Now the previous remark implies that
$$\mathcal{H}|_{\phi_0^{-1} (U_p)} \cong G_p \ltimes \left( G_p \times_{G_q} V_q \right) .$$ Essential injectivity of the resulting orbifold morphism follows directly from Lemma \ref{underiota}. Since the other conditions in Definition \ref{suborbifold} are automatic, we get an orbifold embedding $\phi : \mathcal{H} \to \mathcal{G}$. 
\end{proof}

Section \ref{conv} will be devoted to prove the converse of this proposition.

\section{Hilsum-Scandalis maps}
We give brief review on Hilsum-Skandalis map, and we refer readers to \cite{PS} and \cite{L} for further details.
We first recall  the definition of the action of a orbifold groupoid on manifolds
\begin{definition}
Let $\mathcal{G}$ be an orbifold groupoid. A left $\mathcal{G}$-space is a manifold $E$ equipped with an action by $\mathcal{G}$. Such an action is given by two maps:
\begin{itemize}
\item
an anchor $\pi : E \to G_0$;
\item
an action $\mu : G_1 \times_{G_0} E \to E$.
\end{itemize}
The latter map is defined on pairs $(g,e)$ with $\pi(e)=s(g)$, and written $\mu (g,e) = g \cdot e$. It satisfies the usual identities for an action:
\begin{itemize}
\item
$\pi(g \cdot e) = t(g)$;
\item
$1_x \cdot e = e$;
\item
$g \cdot (h \cdot e) = (gh) \cdot e.$
\end{itemize}
for $x \stackrel{h}{\longrightarrow} y \stackrel{g}{\longrightarrow} z$ in $G_1$ with $\pi(e)=x$.
\end{definition}
A right $\mathcal{G}$-space is the same thing as a left $\mathcal{G}^{op}$-space, where $\mathcal{G}^{op}$ is the opposite groupoid obtained by exchanging the roles of the target and source maps.

\begin{definition}
A left $\CG$-bundle over a manifold $M$ is a manifold $R$ with smooth maps
\begin{equation*}
\begin{diagram}
\node{ R } \arrow{s,r}{r}\arrow{e,t}{\rho} \node{M} \\
\node{ G_0 }
\end{diagram}
\end{equation*}
and a left $\CG$-action $\mu$ on $R$, with anchor map $r: R \to G_0$, such that $\rho(gx) = \rho(x)$ for any $x \in R$ and any $g \in G_1$ with $r(x) = s(g)$.

Such a bundle $R$ is principal if
\begin{enumerate}
\item $\rho$ is a surjective submersion,\\
\item the map $(\pi_1, \mu): R \times_{r, G_0, s} G_1 \to R \times_M R$, sending $(x,g)$ to $(x,gx)$, is a diffeomorphism.
\end{enumerate}
\end{definition}

\begin{definition}
A Hilsum-Scandalis map $\CG \to \CH$ is represented by a principal left $\CH$-bundle $R$ over $G_0$
\begin{equation*}
\begin{diagram}
\node{ R } \arrow{s,r}{r}\arrow{e,t}{\rho} \node{G_0} \\
\node{ H_0 }
\end{diagram}
\end{equation*}
which is also a right $\CG$-bundle (over $H_0$), and the right $\CG$-action commutes with the $\CH$-action. $R$ is called the Hilsum-Scandalis bibundle.
\end{definition}

\begin{definition}
For two bibundles $R : \CG \to \CH$ and $Q : \CH \to \mathcal{K}$, their composition is defined by the quotient of the fiber product $Q \times_{H_0} R$ by the action of $\CH$:
\begin{equation}
	Q \circ R := (Q \times_{H_0} R) / H_1,
\end{equation}
where the action of $H_1$ on $Q \times_{H_0} R$ is given by $h \cdot (q,r) := (qh, h^{-1}r)$.
Since the left action of $\CH$ on $R$ is principal, the action of $\CH$ on $Q \times_{H_0} R$ is free and proper; hence, the $Q \circ R$ is a smooth manifold. It also admits a principal $\mathcal{K}$-bundle structure with a right $\CG$-action, because $\CH$-action commutes with $\CG$- and $\mathcal{K}$-actions on $R$ and $Q$, respectively. 
\end{definition}
One can compose two Hilsum-Scandalis maps as follows:
\begin{definition}
Two Hilsum-Scandalis maps $P, R : \CG \to \CH$ are isomorphic if they are diffeomorphic as left $\CH$- and right $\CG$-bundles: i.e, there is a diffeomorphism $\alpha : P \to R$ satisfying $\alpha(h \cdot p \cdot g) = h \cdot \alpha(p) \cdot g$ for all $(h, p, g) \in H_1 \times_{H_0} P \times_{G_0} G_1$.
\end{definition}

For example, any Lie groupoid homomorphism $\phi : \CG \to \CH$ defines a Hilsum-Scandalis map
\begin{equation*}
\begin{diagram}
\node{ R_\phi := H_1 \times_{s, H_0, \phi_0} G_0} \arrow{s,r}{t \circ \pi_1}\arrow{e,t}{\pi_2} \node{G_0} \\
\node{ H_0 }
\end{diagram}
\end{equation*}
where $\pi_1$ and $\pi_2$ are the projection maps. One can easily check that $\pi_2$ is principal in this case. We will use this construction crucially in the next section to construct an equivariant immersion from an orbifold embedding.

\begin{remark}$\left. \right.$
\begin{enumerate}
\item
Not every Hilsum-Scandalis map is induced from Lie groupoid homomorphisms. In fact, a Hilsum-Scandalis map $R : \CG \to \CH$ is isomorphic to some $R_{\phi}$ for some Lie groupoid homomorphism $\phi : \CG \to \CH$ if and only if the map $\rho : R \to G_0$ has a global section. See Lemma 3.36 in \cite{L}.
\item 
We use slightly different notion of the Hilsum-Skandalis map from \cite{PS}. In \cite{PS},
 $$R_\phi^{ps} = H_0 \times_{\phi_0,G_0,t} G_1$$
is used to construct a Hilsum-Skandalis map from $\phi$. Here, we use $R_\phi^\ast$ to make it a left $G$-space. See the following diagrams.
\begin{equation*}
\xymatrix{ 
   R_\phi \ar[r]^{\pi_1} \ar[d]_{\pi_2}  & G_1 \ar[d]^{s} && R_\phi^{ps} \ar[r]\ar[d] & H_0 \ar[d]^{\phi_0}\\
   H_0 \ar[r]^{\phi_0}   & G_0 && G_1 \ar[r]^{t}   & G_0 } 
\end{equation*}
\end{enumerate}
\end{remark}

Now, we want to translate the notion of equivalence in the category of orbifold groupoids into Hilsum-Scandalis maps. We first refer to the following two lemmas from \cite{L}.

\begin{lemma}(\cite{L}, Lemma 3.34)
A Lie groupoid homomorphism $\phi : \CG \to \CH$ is an equivalence of Lie groupoids if and only if the corresponding $R_\phi$ is $\CG$-principal. 
\end{lemma}

\begin{lemma}(\cite{L}, Lemma 3.37)
Let $P : \CG \to \CH$ be a Hilsum-Scandalis map. Then, there is a cover $\phi : \mathcal{U} \to G_0$ and a groupoid homomorphism $f : \phi^*\CG \to \CH$ so that
\begin{equation*}
P \circ R_{\tilde\phi} \stackrel{\simeq}{\longrightarrow} R_f,
\end{equation*}
where $\tilde\phi : \phi^*\CG \to \CG$ is the induced functor and ``$\stackrel{\simeq}{\longrightarrow}$" an isomorphism of Hilsum-Scandalis maps. Here, $\phi^\ast \CG$ is the Lie groupoid with $\left(\phi^\ast \CG \right)_0 = \mathcal{U}$
and $\left(\phi^\ast \CG \right)_1$ given by
\begin{equation*}
\begin{diagram}
\node{\left(\phi^\ast \CG \right)_1} \arrow{s,l}{(s,t)} \arrow{e} \node{G_1} \arrow{s,r}{(s,t)}\\
\node{\mathcal{U} \times \mathcal{U}} \arrow{e,b}{(\phi,\phi)} \node{G_0 \times G_0}
\end{diagram}
\end{equation*}
\end{lemma}

From the above two lemmas, we obtain the following characterization of Morita equivalence in terms of Hilsum-Scandalis language.

\begin{lemma}
If a Hilsum-Skandalis map $P : \CG \to \CH$ is also right $\CG$-principal, then $f : \phi^*\CG \to \CH$ obtained from the above lemma is an equivalence of groupoids. Note that $\tilde\phi : \phi^*\CG \to \CG$ is trivially an equivalence of groupoids.
\end{lemma}
\begin{proof}
Note that $R_{\tilde\phi}$ is biprincipal, since $\tilde\phi$ is an equivalence of groupoids. The composition of two biprincipal bundle $P\circ R_{\tilde\phi}$ is also biprincipal, and hence the isomorphic bibundle $R_f$ also biprincipal. Therefore $f$ is an equivalence of groupoids $\phi^*\CG$ and $\CH$.
\end{proof}

The above lemma justifies the notion of the Morita equivalence in the Hilsum-Skandalis setting.
\begin{definition}
A Hilsum-Scandalis map $(R, \rho, r)$ is a Morita equivalence when it is both a principal $\CG$-bundle and a principal $\CH$-bundle.
\end{definition}

\section{Construction of equivariant immersions from orbifold embeddings}\label{conv}

Let $\mathcal{H}$ be an orbifold groupoid, and $\phi : \mathcal{H} \to [M/G]$ be a groupoid morphism which factors through an orbifold embedding $\psi$ and a slice representation $p : \CG \to [M/G]$.
\begin{equation}\label{facorbemb}
\xymatrix{
	\CH \ar[dr]_{\psi : orb. emb.} \ar[rr]_{\phi} && [M/G]\\
	& \CG \ar[ur]_{p : \cong} &}
\end{equation}
With this assumption in this section, we construct a $G$-equivariant immersion map $\iota : N \to M$ for some $G$-manifold $N$.

\begin{prop}
Consider $\mathcal{H}, [M/G], \phi, \psi$ as above. 

Then, there exists $G$-space $N$ and a strong $G$-equivariant immersion $\iota : N \to M$ such that
\begin{itemize}
\item $[N/G]$ is Morita equivalent to $\mathcal{H}$.
\item the induced map $[N/G] \to [M/G]$, again denoted by $\iota$, fits into the following diagram of Lie groupoid homomorphisms
\begin{equation}\label{lift}
\xymatrix{ 
   [N/G]  \ar[rr]^{\iota}  \ar[d]_{Morita\, \cong} && [M/G]\\
   \mathcal{H} \ar[urr]_{\phi}   && } 
\end{equation}
and $\iota$ is a $G$-equivariant immersion.
\end{itemize}
\end{prop}

\begin{proof}
From the groupoid structure of $[M/G]$, we have smooth maps
$$\phi_0 : H_0 \to M \qquad \phi_1 : H_1 \to G \times M$$
which are compatible with the structure maps of an orbifold groupoid. As in \cite{PS},  we interpret $\phi$ as a Hilsum-Skandalis-type map. So, we define a bibundle $R_\phi$ as
$$R_\phi := (G \times M) \times_{s,M,\phi_0} H_0.$$
Note that $R_\phi$ is a smooth manifold since $s$ is a submersion.
\begin{equation}\label{HS}
\xymatrix{ 
   R_\phi \ar[r]^{\pi_1} \ar[d]_{\pi_2}  & G \times M \ar[d]^{s} \\
   H_0 \ar[r]^{\phi_0}   & M} 
\end{equation}

This space is first of all smooth and has two maps to $H_0$ and $M$,
$$H_0 \stackrel{\pi_2}{\longleftarrow} R_{\phi} \stackrel{\WT{\iota} := t \circ \pi_1}{\longrightarrow} M$$ 
which will be used as anchor maps below. (We denote $t \circ \pi_1$ by $\WT{\iota}$.) 

We define a right $\mathcal{H}$-action and a left $G$-action on $R_\phi$ as follows:
Write an element of $R_\phi$ by $\left(g, \phi_0 (y), y \right)$ which indicates a point $y$ in $H_0$ and an arrow $g \in G \times M$ whose source is $\phi_0 (y)$. Then,
\begin{itemize}
\item For an arrow $h \in H_1$,
$$  \left(g, \phi_0 (y), y \right) \cdot h := \left( g \circ \phi_1(h), \phi_0(s(h)), s(h)   \right);$$
\item For $g \in G$,
$$g' \cdot \left(g, \phi_0 (y), y \right) := \left( g' \circ g, \phi_0 (y), y \right).$$ 
\end{itemize}
$R_\phi$ is a right $\mathcal{H}$-space and a left $G$-space as the figure \ref{twoactions} below shows. As mentioned, the corresponding anchor maps are $\pi_2$ and $\tilde{\iota}$, respectively. Indeed, $\pi_2$ is a principal left $G$-bundle ($\pi_2$ is a submersion since $s$ in the diagram \eqref{HS} is submersion).

\begin{figure}[h]
\begin{center}
\includegraphics[height=2in]{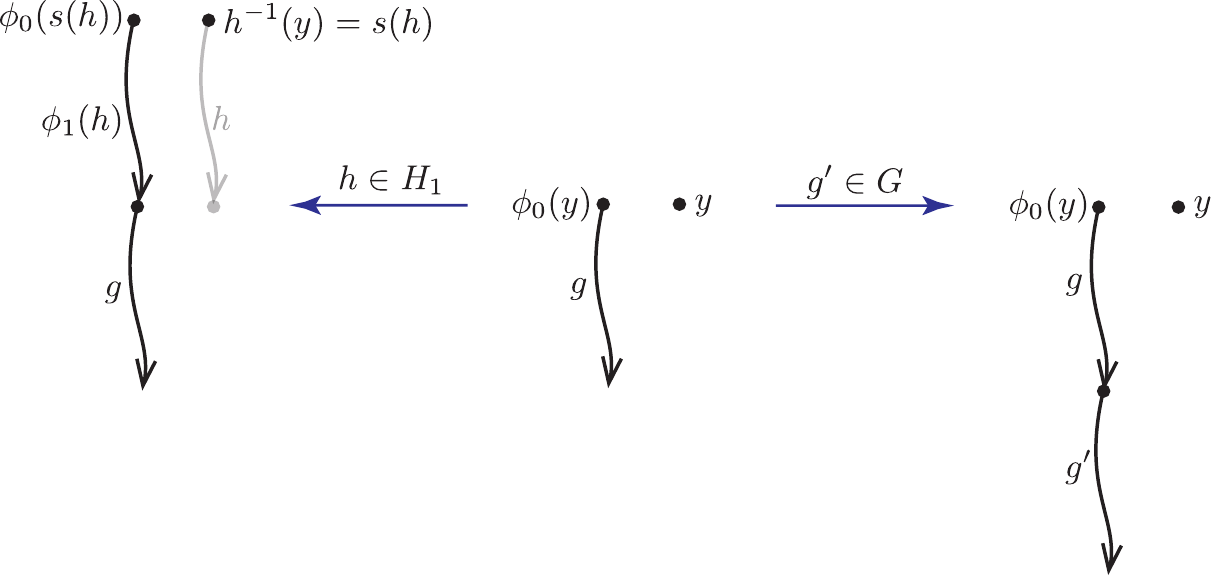}
\caption{the right $\mathcal{H}$-action and the left $G$-action on $R_\phi^\ast$}
\label{twoactions}
\end{center}
\end{figure}
Now, the following are clear from the definition of both actions.

\begin{lemma}\label{twoact}
Two actions defined above have the following properties:
\begin{enumerate}
\item[(1)] The right $\mathcal{H}$-action we have defined is free;
\item[(2)] The left $G$-action and the right $\mathcal{H}$-action commute;
\item[(2)] $\WT{\iota}$ is a $G$-equivariant map which is invariant under the $\mathcal{H}$-action.
\end{enumerate}
\end{lemma}
\begin{proof}
We only show (1) and the others follow from the definition. Suppose $h$ fixes $\left((g,\phi_0 (y)),y \right) \in R_\phi$. Then $h$ should be an element of $H_y$ and $g \cdot \phi_1(h) = g$, where $H_y$ is a local isotropy group of $y \in V_y$ for some local chart $H_y \ltimes V_y$ of $\CH$. Thus $h$ lies in the kernel of the group homomorphism 
$$\phi_1|_{H_y} : H_y \to G.$$
Note that equivalence map $p : \CG \to [M/G]$ preserves isotropy groups. More precisely,
$$Hom_{\CG} (x,z) \cong Hom_{[M/G]} (p(x), p(z))$$
for all $x,z \in G_0$.
Now it follows that $\phi_1 | = p_1 \circ \psi_1 |$ is injective, because $\psi_1 |$ is injective from the definition of orbifold embedding and $p_1$ preserves isotropy groups. Hence, the $h$ is the identity.
\end{proof}

We denote by $N$ the quotient space of $R_\phi$ by the right $\mathcal{H}$-action and by $\pi : R_\phi^\ast  \to N$ the quotient map. 

\begin{lemma}
$N$ is a smooth manifold.
\end{lemma}
\begin{proof}
This follows directly from the fact that the right $\mathcal{H}$-action on $R_\phi$ is free ((2) of Lemma \ref{twoact}) and proper (because $\CH$ itself is \'{e}tale and hence proper). 
\end{proof}

From (2) of Lemma \ref{twoact}, $N$ admits a left $G$-action which is induced by $G$-action on $R_\phi$. Since the $G$-action on $M$ is locally free, so is it on $N$. Therefore, we get a global quotient  orbifold $[N/G]$ from the orbifold embedding $\phi:\mathcal{H} \to [M/G]$.

\begin{lemma}
$[N/G]$ is Morita equivalent to $\mathcal{H}$.
\end{lemma}
\begin{proof}
Note that we have a Hilsum-Skandalis map $ \mathcal{H} \to [N/G] $ (or, $\bar{\pi}_2 : [N/G] \to \mathcal{H}$ from $\pi_2 : R_\phi \to H_0$):
\begin{equation*}
\xymatrix{ 
   R_\phi  \ar[d]_{\pi}  \ar[r]^{\pi_2} & H_0\\
   N   &  } 
\end{equation*}
We have shown that $\pi$ is principal in (2) of Lemma \ref{twoact}. It is also obvious from the Hilsum-Skandalis construction that $\pi_2$ is principal. So the Hilsum-Skandalis map from $R_\phi$ is a Morita equivalence.
\end{proof}

From (3) of Lemma \ref{twoact}, we can observe that $\WT{\iota}$ factors through the quotient space $H$. Since $\WT{\iota}$ is a $G$-equivariant, we get a $G$-equivariant map $\iota : N \to M$. Furthermore,
\begin{lemma}
$\iota$ is a $G$-equivariant immersion.
\begin{equation}
\xymatrix{ 
   R_\phi  \ar[d]_{\pi}  \ar[rr]^{\WT{\iota}} && M\\
   N \ar[urr]_{\iota}   &&   } 
\end{equation}
\end{lemma}

\begin{proof}
Since being an immersion is a local property, it suffices to prove it locally. However, we have a nice local model of $\iota$ from \eqref{model}. Thus, it is enough to prove it with $\CH = G_x \ltimes (G_x \times_{H_y} V_y)$, $\CG =  G_x \ltimes U_x$, $\psi : G_x \ltimes (G_x \times_{H_y} V_y) \to G_x \ltimes U_x $ and $p : G_x \ltimes U_x \to G \ltimes \WT{U}_x$ for $\WT{U}_x \cong G \times_{G_x} U_x$ as in the Lemma \ref{Slice theorem}.



Then, $R_\phi = \left(G \times \WT{U}_x \right) \times_{s, \WT{U}_x, \phi_0} \left( G_x \times_{H_y} V_y \right)$. We mod it out by the right $G_x$-action (considered as a local $H_1$-action) to get the local shape of $N$, again denoted by $N$ in this proof. Recall that this $G_x$-action is given by
$$\left( (k_1,a), [g_1 ,b] \right) \cdot g = [(k_1 g, g^{-1} \cdot a)  , [g^{-1} g_1, b]] $$
for $g \in G_x$, $k_1 \in G$ and $g_1 \in G_x$ where $a = g_1 \cdot \phi_0 (b)$. And $\tilde{\iota}$ on $\mathcal{R}_\phi$ which projects down to $\iota$ on $N$ is defined as
\begin{equation}\label{iotabarlocal}
\tilde{\iota} \left( (k_1,a), [g_1 ,b] \right)  = k_1 \cdot a = k_1 g_1 \cdot \phi_0 (b) \in \WT{U}_x.
\end{equation}
where $\phi  = p \circ \psi  : \mathcal{H} \to [M/G]$.

For given $z \in \WT{U}_x$, we check how many points in $N$ are mapped to $z$. Suppose
$$\iota  [(k_1,a_1), [g_1 ,b_1] ] =z \quad \mbox{and} \quad
\iota  [(k_2,a_2), [g_2 ,b_2] ] =z.$$
Up to the $G_x$-action, we may assume that $g_1=g_2=1$ (recall $N = G_x \setminus R_\phi $). Therefore, we have $k_1 \phi_0 (b_1) = k_2 \phi_0 (b_2) = z$ by \eqref{iotabarlocal}. This implies that $k_2^{-1} k_1 \in G_x$ since both $\phi_0 (b_1)$ and $\phi_0 (b_2)$ belong to the normal slice at $x$. As $G_x$ is finite, there are finitely many $k_2$ with this property. 

Since $p_0$ is an embedding and every fiber of $\psi_0$ is finite, $\phi_0 = p_0 \circ \psi_0$ is an immersion whose fibers are all finite as well. Finally, as $b_2$ lies in the fiber $\phi_0^{-1} ( k_2^{-1} k_1 b_1)$, there can exist only finitely many such $b_2$'s.

\end{proof}
Finally, we show  in the following lemma that the resulting equivaraint immersion is strong which finishes the proof of the proposition
\end{proof}
\begin{lemma}\label{strong}
$\iota : N \to M$ constructed above is strong.
\end{lemma}

\begin{proof}
Note that $|N/G| \cong |\mathcal{H}|$ and $|M/G| \cong |\mathcal{G}|$. From the construction in Section \ref{generalcase} (or Section \ref{conv}), we have
\begin{equation*}
\xymatrix
{ |N/G| \ar[r]^{|\iota|} \ar[d]_{\cong} & |M/G| \ar[d]^{\cong}\\
 |\mathcal{H}| \ar[r]_{|\phi|} & |\mathcal{G}|}
\end{equation*}
Since $|\phi|$ is injective from the definition of the orbifold embedding, $|\iota|$ is injective.

\end{proof}

\section{General case}\label{generalcase}
So far, we have considered a translation groupoid $[M/G]$ as our target space. The
construction can be generalized to the case of general orbifolds which we will discuss from now on. We state this as a theorem, first.

\begin{theorem}\label{moregene}
Let $\phi : \CH \to \CG$ be an orbifold embedding, where $\CG$ is Morita equivalent to a translation groupoid $[M/G]$. Then, there exist a manifold $N$ on which the Lie group $G$ acts locally freely such that
\begin{enumerate}
\item[(i)] $\mathcal{H} \simeq [N/G]$ and $\mathcal{G} \simeq [M/G]$
\item[(ii)] there exists a $G$-equivariant immersion $\iota : N \to M$ which makes the diagram
\begin{equation}\label{Mor}
\xymatrix{ 
   \mathcal{H}  \ar@{<->}[d]_{Morita}  \ar[rr]^{\phi} && \mathcal{G} \ar@{<->}[d]^{Morita}\\
   [N/G] \ar[rr]_{\iota}   &&  [M/G]  } 
\end{equation}
commute.
\end{enumerate}
\end{theorem}

\begin{remark}
The diagram in (ii) of the theorem can be regarded as a diagram of morphisms in the category of Lie groupoids where we can invert equivalances. (See Definition \ref{def:Mor}, or \cite{ALR} for the precise definition of the morphisms in the category of groupoids.)
\end{remark}
We proceed the poof of theorem \ref{moregene} as follows:\\


After fixing a Morita equivalence map 
$$\mathcal{G} \stackrel{\psi \,\simeq}{\longleftarrow} \mathcal{G}' \stackrel{\sigma \, \simeq}{\longrightarrow} [M/G]$$
for some Lie groupoid $\mathcal{G}'$, we pull back the equivalence map $\psi : \mathcal{G}' \to \mathcal{G}$ to $\mathcal{H}$ to get $\phi^\ast \mathcal{G}'=\CH \times_{\CG} \mathcal{G}'$. Recall that
\begin{align*}
(\phi^\ast \mathcal{G}')_0 &= H_0 \times_{\phi_0, G_0, s} G_1 \times_{t, G_0, \psi_0} G_0'\\
(\phi^\ast \mathcal{G}')_1 &= H_1 \times_{s\phi_1, G_0, s} G_1 \times_{t, G_0, s\psi_1} G_1'
\end{align*}
We denote the composition $\sigma \circ pr_2 : \phi^\ast \mathcal{G}' \to [M/G]$ by $\tilde{\phi}$. Then,
$$\tilde{\phi}_0 = \sigma_0 \circ (pr_2)_0 \quad \tilde{\phi}_1 = \sigma_1 \circ (pr_2)_1$$
where $pr_2$ is the projection from $\phi^\ast \mathcal{G}'=\CH \times_{\CG} \mathcal{G}'$ to $\mathcal{G}'$. We will apply the construction in the previous section to $\tilde{\phi}$.

\begin{equation}\label{secondstage}
\xymatrix{
& & [M/G]  \\
&&\\
\phi^\ast \mathcal{G}' \ar[rr]^{pr_2} \ar[uurr]^{\tilde{\phi}} \ar[dd]_{pr_1}& & \mathcal{G}' \ar[uu]_{\sigma : equiv.}\ar[dd]^{\psi : equiv.} \\
&&\\
\mathcal{H} \ar[rr]_{\phi : orb. emb.}&& \mathcal{G} }
\end{equation}
First of all $\phi^\ast \CG'$ is equivalent to $\CH$. Pull-back of any equivalence is again an equivalence as shown in Lemma \ref{pbequiv}.


To construct a $G$-equivariant immersion from $\tilde{\phi} : \phi^\ast \mathcal{G}' \to [M/G]$, we introduce the Hilsum-Skandalis bibundle associated to $\tilde{\phi}$ as we did in the previous section. Recall
$$R_{\tilde{\phi}} = (G \times M) \times_{s,M,\tilde{\phi}_0} (\phi^\ast \mathcal{G}')_0.$$
An element of $R^\ast_{\tilde{\phi}}$ consists of the following data.
\begin{equation*}
\xymatrix{\stackrel{m}{\odot} \ar[d]^a & \star \left( = [x, \,\,\, \phi_0 (x) \stackrel{g}{\longrightarrow}  \psi_0 (z), \,\,\, z] \right) \\ &  }
\end{equation*}
where $m \in M$, $a \in G$, $x \in H_0$, $z \in G_0'$ and $\sigma_0 (z) =m$. Write $\bf{r}$ for this element. Then, $pr_2 ({\bf r}) =z$ and the $G$-equivariant map $\tilde{\iota} :  R_{\tilde{\phi}} \to M$ is given by $\tilde{\iota} ({\bf r}) = a \cdot \sigma_0 (z)$. Recall
\begin{equation}\label{generaliota}
\xymatrix{R_{\tilde{\phi}} \ar[d] \ar[r]^{\tilde{\iota}}& M \\
N \ar[ur]_{\iota} & 
}
\end{equation}
where $N$ is obtained from $R_{\tilde{\phi}}$ after taking a quotient by $\phi^\ast \mathcal{G}'$-action. Since local groups are preserved by equivalences, restriction of 1-level maps appearing in \ref{secondstage} to any local groups are all injective. Then, similar argument as in Lemma \ref{twoact} shows the right $\phi^\ast \mathcal{G}’$-action on $R_{\tilde{\phi}}$ is free and proper. Note that $\phi^\ast \CG$ is proper since it is equivalent to the proper Lie (indeed, \'{e}tale) groupoid $\CG$. Therefore, $N$ is a smooth manifold.

It remains to show that the induced $G$-equivariant map $\iota : N \to M$ is indeed an immersion. We will directly compute the kernel of $d\tilde{\iota}$. For notational simplicity, we will write $\tau_\ast$ for the derivative $d \tau$ of a smooth map $\tau$ between two manifolds.

 A tangent vector on $R_{\tilde{\phi}} $ at $\bf{r}$ is given by the tuple
$${\bf v} = [(v_l=v_a \oplus v_m,v_m),v_x,v_g,v_z]$$
where $v_l \in T(G \times M)$ and $v_a \in TG$ with the relations
\begin{itemize}
\item $s_\ast(v_l)=v_m$
\item $\left(\sigma_0\right)_\ast (v_z) = v_m$,
\item $s_\ast (v_g) =  \left(\phi_0\right)_\ast (v_x) \quad  t_\ast (v_g) = \left( \psi_0 \right)_\ast (v_z) $.
\end{itemize}
Since $\mathcal{G}$ is \'etale and hence both
$$s_\ast : T_g G_1 \to T_{\phi_0 (x)} G_0 \quad\mbox{and} \,\,\,\,\,\, t_\ast : T_g G_1 \to T_{\psi_0 (z)} G_0$$
are isomorphisms, we may rewrite the third relation as  
\begin{equation}\label{thirdrel}
\,\,\, \bullet \,\, s_\ast^{-1} \left(\phi_0\right)_\ast (v_x) = t_\ast^{-1} \left( \psi_0 \right)_\ast (v_z) = v_g.
\end{equation}

From the first relation, it suffices to represent ${\bf v}$ as
$${\bf v} = [v_a \oplus v_m,v_x,v_g,v_z]$$
Note that $v_m$ and $v_g$ are determined by $v_x$ and $v_z$. One can easily check that
$$\tilde{\iota}_\ast ({\bf v}) = t_\ast (v_a \oplus v_m) =  \left(v_a\right)^{\#} +(L_a)_* (v_m),$$
where $(v_a)^{\#}$ is a vector field on $M$ generated by the infinitesimal action of $v_a$ on $M$.
For simplicity, we assume that $a$ is the identity element of $G$. Then,
$$\tilde{\iota}_\ast ({\bf v}) = (v_a)^{\#} + v_m.$$
Our goal is to compute the kernel of this map. If we can show that the kernel of $\tilde{\iota}_\ast$ lies in the tangent direction of $(\phi^\ast \mathcal{G}')$-orbit, then it will imply that $\iota$ is an immersion.

To do this, we first characterize the direction of $(\phi^\ast \mathcal{G}')$-orbit. By the definiton of the right $\phi^\ast \mathcal{G}'$ action on $R_{\tilde{\phi}}$, we should consider an arrow of  $\phi^\ast \mathcal{G}'$ given by a pair $(h,k) \in H_1 \times G_1'$ such that 
\begin{equation}\label{fixtar}
t(h) = x \,\,\, \mbox{and}\,\,\, t(k)=z.
\end{equation} Considering the infinitesimal version of $\phi^\ast \mathcal{G}'$-action carefully, we have the following lemma.
\begin{lemma}
${\bf v}$ is tangent to $(\phi^\ast \mathcal{G}')$-orbit if and only if there exists $(v_h, v_k) \in T_h H_1 \times T_k G_1'$ such that 
\begin{equation}\label{inffixtar}
t_\ast (v_h) = t_\ast (v_k) =0
\end{equation}
and
\begin{equation}\label{orbitdir}
{\bf v} = [  (\sigma_1)_\ast (v_k), s_\ast (v_h),v_g, s_\ast (v_k) ].
\end{equation}
(Here, we do not specify $v_g$ since they are completely determined by other components \eqref{thirdrel}.)
\end{lemma}

\begin{remark}
In the equation (\ref{inffixtar}), $t_\ast (v_h)=0$ implies $v_h =0$ since $\mathcal{H}$ is \'etale.
Then the equation (\ref{orbitdir}) can be rewritten as
\begin{equation}\label{orbitdir1}
{\bf v} = [  (\sigma_1)_\ast (v_k), 0, 0, s_\ast (v_k) ].
\end{equation}
\end{remark}

Now, we are ready to prove the desired property of $\iota$.

\begin{lemma}
$\iota : N \to M$ above is an immersion.
\end{lemma}
\begin{proof}
Let ${\bf v} = [(v_a,v_m),v_x,v_g,v_z] \in \ker \tilde{\iota}_\ast$. i.e. $t_\ast (v_a \oplus v_m)=(v_a)_{\#} + v_m =0$. We should show that $v_h = 0$ and find $v_k$ satisfying \eqref{inffixtar} and \eqref{orbitdir1}.
First, we find $v_k$ as follows:\\
From the condition $\left(\sigma_0\right)_\ast (v_z) = v_m = (-v_a)_{\#}$, we get
$$\exp(t \cdot (-v_a)) \cdot m = m(t),$$
where $m(t)$ is a curve in $M$ with $m'(0) = v_m$.
Since $\exp{(-t \cdot v_a)} = \exp{(t \cdot v_a)}^{-1}$,
$$m \equiv \exp{(t \cdot v_a)} \cdot m(t).$$
Note that $m(t) = \sigma_0 (z(t))$ for some curve $z(t)$ in $G'_0$ with $z'(0) = v_z$. Since $\sigma$ is an equivalence map and $(a(t),m(t)) \in G \times M$ is an arrow from $\sigma_0 (z(t))$ to $\sigma_0 (z)$ for each $t$, there is unique $k(t) \in G_1'$ which maps to $\sigma_1 (k(t)) = (a(t),m(t))$ with $s(k(t))=z(t)$ and $t(k(t))\equiv z$. We define $v_k := k'(0) \in T_k G_1'$, then $(\sigma_1)_\ast (v_k) =(v_a, v_m) \in TG \times TM$ and $s_\ast (v_k) = v_z$.

Let $\gamma$ be a curve in the $R^\ast_{\tilde{\phi}}$ such that $\gamma'(0) = {\bf v}$. Consider a component of $\gamma$, 
 $g(t) \in G_1$ such that $g'(0) = v_g$. We claim that this curve $g(t)$ is a constant curve.

Note that $\psi_0 \circ t \circ k(t) \equiv \psi_0 (z)$. Since $\psi$ is an equivalence map and $\CG$ is \'etale, $\psi_0 \circ s \circ k(t) \equiv \psi_0 (z)$($\psi_1$ maps ``infinitesimal action'' on $\CG'$ whose image of target is fixed to a ``constant action'' on $\CG$).
Note that $t \circ g(t) = \psi_0 \circ z(t) = \psi_0 \circ s \circ k(t) = \psi_0(z)$. Since $\CG$ is \'etale and target points of $g(t)$ is fixed, $g(t)$ is a constant arrow in $G_1$.
Since $\phi_0 \circ x(t) = s \circ g(t) = \phi_0 (x)$ and $\phi_0$ is an immersion map, $x(t) \equiv x$. 

We conclude that, if $\tilde\iota_\ast({\bf{v}})=0$, then there exist $v_k$ such that
\begin{equation}
{\bf v} = [  (\sigma_1)_\ast (v_k), s_\ast (0)=0,v_g=0, s_\ast (v_k) ],
\end{equation}
it proves that $\iota : N \to M$ is an immersion.
\end{proof}

Lastly, the equivariant immersion $\iota$ is strong by basically the same argument as in Lemma \ref{strong}.

\bibliographystyle{amsalpha}

\end{document}